\documentclass[a4paper, 11pt, twoside]{article}

\usepackage{amsmath, amscd, amsfonts, amssymb, amsthm, latexsym, url, color, todonotes, bm, framed}

\usepackage{todonotes}
\usepackage{graphicx}
\usepackage[left=1.3in, right=1.3in, top=1.5in, bottom=1.5in, includefoot, headheight=13.6pt]{geometry}

\input{xy}
\xyoption{all}
\xymatrixrowsep{7mm}
\xymatrixcolsep{7mm}

\usepackage[ps2pdf=true,colorlinks,breaklinks=true]{hyperref}
\usepackage[figure,table]{hypcap}
\hypersetup{
	bookmarksnumbered,
	pdfstartview={FitH},
	citecolor={black},
	linkcolor={black},
	urlcolor={black},
	pdfpagemode={UseOutlines}
}
\makeatletter
\newcommand\org@hypertarget{}
\let\org@hypertarget\hypertarget
\renewcommand\hypertarget[2]{%
  \Hy@raisedlink{\org@hypertarget{#1}{}}#2%
} 
\makeatother

\newtheorem{theorem}{Theorem}[section]
\newtheorem{lemma}[theorem]{Lemma}
\newtheorem{corollary}[theorem]{Corollary}
\newtheorem{proposition}[theorem]{Proposition}

\theoremstyle{definition}
\newtheorem{definition}[theorem]{Definition}
\newtheorem{remark}[theorem]{Remark}
\newtheorem{example}[theorem]{Example}

\newcommand{\xysquare}[8]{
\[\xymatrix{
#1 \ar@{#5}[r] \ar@{#6}[d] & #2 \ar@{#7}[d]\\
#3 \ar@{#8}[r] & #4
}\]
}

\DeclareMathOperator*{\projlimf}{``\varprojlim''}
\DeclareMathOperator*{\holim}{\operatorname*{holim}}

\newcommand{\al}{\alpha}
\newcommand{\bb}{\mathbb}

\newcommand{\blob}{\bullet}

\newcommand{\comment}[1]{}

\newcommand{\dotimes}{\otimes^{\sub{\tiny I\hspace{-1.3mm}L}}}

\newcommand{\into}{\hookrightarrow}
\newcommand{\isoto}{\stackrel{\simeq}{\to}}
\newcommand{\Isoto}{\stackrel{\simeq}{\longrightarrow}}
\newcommand{\Mapsto}{\longmapsto}

\newcommand{\onto}{\twoheadrightarrow}
\newcommand{\op}{\operatorname}
\newcommand{\pid}[1]{\langle #1 \rangle}
\newcommand{\quis}{\stackrel{\sim}{\to}}
\newcommand{\res}{\overline}
\newcommand{\roi}{\mathcal{O}}

\newcommand{\sub}[1]{{\mbox{\scriptsize #1}}}

\newcommand{\To}{\longrightarrow}

\newcommand{\xto}{\xrightarrow}

\newcommand{\THH}{T\!H\!H}
\newcommand{\HH}{H\!H}
\newcommand{\HC}{H\!C}

\newcommand{\HP}{H\!P}
\newcommand{\TR}{T\!R}
\newcommand{\TC}{T\!C}
\newcommand{\CC}{C\!C}

\renewcommand{\cal}{\mathcal}
\renewcommand{\hat}{\widehat}
\renewcommand{\frak}{\mathfrak}
\newcommand{\indlim}{\varinjlim}
\renewcommand{\tilde}{\widetilde}

\renewcommand{\ker}{\operatorname{Ker}}
\renewcommand{\projlim}{\varprojlim}

\DeclareMathOperator{\Ann}{Ann}

\DeclareMathOperator{\Hom}{Hom}

\DeclareMathOperator{\Tor}{Tor}

\usepackage{fancyhdr}

\usepackage{sectsty}
\sectionfont{\Large\sc\centering}
\chapterfont{\large\sc\centering}
\chaptertitlefont{\LARGE\centering}
\partfont{\centering}

\begin{document}

\title{\vspace{-6mm}Pro unitality and pro excision in algebraic $K$-theory and cyclic homology}

\author{Matthew Morrow}

\date{}

\maketitle
\vspace{-5mm}
\centerline{{\em Journal f\"ur die reine und angewandte Mathematik}, to appear.}

\begin{abstract}
The purpose of this paper is to study pro excision in algebraic $K$-theory and cyclic homology, after Suslin--Wodzicki, Cuntz--Quillen, Corti\~nas, and Geisser--Hesselholt, as well as continuity properties of Andr\'e--Quillen and Hochschild homology. A key tool is first to establish the equivalence of various pro Tor vanishing conditions which appear in the literature.

This allows us to prove that all ideals of commutative, Noetherian rings are pro unital in a suitable sense. We show moreover that such pro unital ideals satisfy pro excision in derived Hochschild and cyclic homology. It follows hence, and from the Suslin--Wodzicki criterion, that ideals of commutative, Noetherian rings satisfy pro excision in derived Hochschild and cyclic homology, and in algebraic $K$-theory.

In addition, our techniques yield a strong form of the pro Hochschild--Kostant--Rosenberg theorem; an extension to general base rings of the Cuntz--Quillen excision theorem in periodic cyclic homology; a generalisation of the Fe\u\i gin--Tsygan theorem;
a short proof of pro excision in topological Hochschild and cyclic homology; and new Artin--Rees and continuity statements in Andr\'e--Quillen and Hochschild homology.

MSC: 19D55 (primary), 16E40 13D03 (secondary).
\end{abstract}

\tableofcontents

\setcounter{section}{-1}
\section{Introduction and statements of main results}
We begin with some remarks on excision in algebraic $K$-theory. It has been known at least since work by R.~Swan \cite{Swan1971} that $K$-theory fails to satisfy excision; i.e., if $A\to B$ is a homomorphism of rings and $I$ is an ideal of $A$ mapped isomorphically to an ideal of $B$, then the map $K_n(A,I)\to K_n(B,I)$ of relative $K$-groups need not be an isomorphism if $n>0$. Having fixed $I$ as a non-unital algebra, A.~Suslin \cite{Suslin1995} showed, by building on earlier work of himself and M.~Wodzicki \cite{Suslin1992}, that $I$ satisfies excision for all such homomorphisms $A\to B$ if and only if $I$ is {\em Tor-unital} in the sense that $\op{Tor}^{\bb Z\ltimes I}_n(\bb Z,\bb Z)=~0$ for $n>0$. Unfortunately, this is not commonly satisfied for rings of algebraic geometry. A recent trend has therefore been to consider instead the problem of ``pro excision'', namely to determine when the map of pro abelian groups $\{K_n(A,I^r)\}_r\to\{K_n(B,I^r)\}_r$ is an isomorphism for all $n\ge0$. 

In particular, a theorem of T.~Geisser and L.~Hesselholt \cite{GeisserHesselholt2006}, whose rational version is due to G.~Corti\~nas \cite{Cortinas2006}, states that pro Tor-unital ideals satisfy pro excision in algebraic $K$-theory, thereby offering a pro version of the aforementioned Suslin--Wodzicki criterion; here we have adopted the following piece of terminology:

\begin{definition}\label{definition_pro_H}
A non-unital ring $I$ is {\em pro Tor-unital} if and only if $\{\Tor_n^{\bb Z\ltimes I^r}(\bb Z,\bb Z)\}_r=~0$ for all $n>0$.
\end{definition}

Unfortunately, until now it has appeared to be difficult to verify whether any given ideal is pro Tor-unital, or to offer many examples of such ideals, owing to the inaccessibility of the rings $\bb Z\ltimes I^r$; hence Geisser--Hesselholt's criterion has been hard to apply in concrete situations. Our first main theorem, the proof of which is the content of Section \ref{section_proof_of_1}, overcomes this difficulty:

\begin{theorem}[See Thm.~\ref{theorem1a}]\label{theorem1}
Let $k$ be a commutative ring, $A\to B$ a homomorphism of $k$-algebras, and $I$ an ideal of $A$ mapped isomorphically to an ideal of $B$. Then the following pro Tor vanishing conditions are equivalent:
\begin{enumerate}
\item $\{\Tor_n^A(A/I^r,A/I^r)\}_r=0$ for all $n>0$.
\item $\{\Tor_n^B(B/I^r,B/I^r)\}_r=0$ for all $n>0$.
\item $\{\Tor_n^{k\ltimes I^r}(k,k)\}_r=0$ for all $n>0$.
\item $I$ is pro Tor-unital, i.e.,  $\{\Tor_n^{\bb Z\ltimes I^r}(\bb Z,\bb Z)\}_r=0$ for all $n>0$.
\end{enumerate}
\end{theorem}

Theorem \ref{theorem1} is a fundamental tool used throughout the paper, which unifies different conditions naturally appearing in excision theory. Condition (i) will be seen to be the key property required for pro-excision in derived Hochschild and cyclic homology, their topological counterparts, and Andr\'e--Quillen homology. When $k$ is a field, condition (iii) is equivalent to the existing notion of H-unitality of the pro $k$-algebra $I^\infty$ (see Eg.~\ref{example_pro_H}), which is central in Wodzicki's original approach to excision \cite{Wodzicki1989}, as well as in J.~Cuntz and D.~Quillen's approach to excision in periodic cyclic homology \cite{CuntzQuillen1997}. The importance of (iv) is thus not only its relevance to pro excision in $K$-theory, but also that it reveals that conditions (i)--(iii) are intrinsic properties of the non-unital ring $I$, depending neither on the ring $A$ nor on the algebra structure from the base ring $k$.

The first concrete application of Theorem \ref{theorem1} is to commutative rings: if $I$ is an ideal of a commutative, Noetherian ring $A$, then M.~Andr\'e \cite{Andre1974} noted that condition (i) is always true (see Lem.~\ref{lemma_Andre}), and so we obtain:

\begin{theorem}[See Thm.~\ref{theorem_pro_Tor_unital_commutative}]\label{theorem_comm_Noeth}
Let $I$ be an ideal of a commutative, Noetherian ring. Then $I$ is pro Tor-unital.
\end{theorem}

Applying Geisser--Hesselholt's aforementioned pro version of the Suslin--Wodzicki criterion, we have the following consequence of Theorem \ref{theorem_comm_Noeth} which completely solves the pro excision problem in $K$-theory for commutative, Noetherian rings:

\begin{corollary}[See Corol.~\ref{corollary2a}]\label{corollary2}
Ideals of commutative, Noetherian rings satisfy pro excision in algebraic $K$-theory. In particular, if $A\to B$ is a homomorphism of commutative, Noetherian rings, and $I$ is an ideal of $A$ mapped isomorphically to an ideal of $B$, then the map of pro abelian groups \[\{K_n(A,I^r)\}_r\To \{K_n(B,I^r)\}_r\] is an isomorphism for all $n\in\bb Z$.
\end{corollary}

Now we turn to Hochschild and cyclic homology. Just as for $K$-theory in the first paragraph of the Introduction, these homology theories do not in general satisfy excision. To further justify the usefulness of pro Tor-unitality, we prove that it is sufficient to ensure that an ideal satisfy pro excision in the derived versions of these homology theories:

\begin{theorem}[Pro excision for derived $\HH$ and $\HC$ {[Thm.~\ref{theorem_pro_excision_for_HH_HC}]}]\label{theorem4}
Let $k$ be a commutative ring, $A\to B$ a homomorphism of $k$-algebras, and $I$ a pro Tor-unital ideal of $A$ mapped isomorphically to an ideal of $B$. Then the canonical maps of pro relative groups \[\{\HH^k_n(A,I^r)\}_r\To \{\HH^k_n(B,I^r)\}_r,\quad\quad \{\HC^k_n(A,I^r)\}_r\To\{\HC^k_n(B,I^r)\}_r\] are isomorphisms for all $n\ge0$.
\end{theorem}

There are two important remarks to make now. Firstly, thanks to Theorem \ref{theorem_comm_Noeth}, the pro Tor-unitality assumption in Theorem \ref{theorem4} is automatically satisfied as soon as $A$ or $B$ is commutative and Noetherian. Secondly, the Hochschild and cyclic homology appearing in Theorem \ref{theorem4}, and elsewhere in the paper, is always defined in a derived sense, that is after replacing rings by free simplicial resolutions (see \S\ref{section_HH_and_HC} for details); if $k$ is a field, then this is the usual Hochschild and cyclic homology. 

There are two key components to the proof of Theorem \ref{theorem4}. The first is a restriction spectral sequence relating the Hochschild homologies of $A$ and $B$; this is given in Proposition \ref{proposition_HH_SS}, and is used in the ensuing corollaries to show that the obstruction to pro excision is largely measured by the pro Tor groups $\{\Tor_n^A(A/I^r,A/I^r)\}_r$ and $\{\Tor_n^B(B/I^r,B/I^r)\}_r$. The second component to the proof is Theorem \ref{theorem1}, which proves that these obstructions vanish since $I$ is pro Tor-unital.

In Section \ref{section_converse} we show moreover that the expected converse statements are true: ideals which satisfy pro excision in Hochschild homology are necessarily pro Tor-unital.

These arguments for Hochschild homology may be repeated verbatim for topological Hochschild homology to give a short proof that pro Tor-unital ideals satisfy pro excision in topological Hochschild and cyclic homology; this was proved originally by Geisser and Hesselholt \cite[Thm.~2.2]{GeisserHesselholt2006} via a lengthy argument:

\begin{theorem}[Pro excision for $\THH$ and $\TC$ {[Thm.~\ref{theorem_pro_THH}]}]\label{theorem5}
Let $A\to B$ be a homomorphism of rings, and $I$ a pro Tor-unital ideal of $A$ mapped isomorphically to an ideal of $B$. Then the canonical maps \[\{\THH_n(A,I^r)\}_r\To \{\THH_n(B,I^r)\}_r,\quad\quad\{\TC_n^m(A,I^r;p)\}_r\To \{\TC_n^m(B,I^r;p)\}_r\] of pro abelian groups are isomorphisms for all $n\ge0$, $m\ge1$, and primes $p\ge2$.
\end{theorem}

In Section \ref{subsection_HP} we turn our attention to periodic cyclic homology. In stark contrast to algebraic $K$-theory, Hochschild, and cyclic homology, it was proved by Cuntz and Quillen \cite{CuntzQuillen1997} that periodic cyclic homology $\HP^k$ does satisfy excision whenever the base ring $k$ is a characteristic zero field. We extend this celebrated result, using derived periodic cyclic homology, to arbitrary commutative base rings of characteristic zero:

\begin{theorem}[{Excision for $\HP^k$ [Thm.~\ref{theorem_Cuntz-Quillen}]}]\label{theorem_intro_CQ}
Let $k$ be a commutative $\bb Q$-algebra. Then derived period cyclic homology over $k$ satisfies excision; i.e., if $A\to B$ is a homomorphism of $k$-algebras, and $I$ is an ideal of $A$ mapped isomorphically to an ideal of $B$, then the canonical map \[\HP_n^k(A,I)\To\HP_n^k(B,I)\] is an isomorphism for all $n\in\bb Z$.
\end{theorem}

The remaining results of the paper focus on commutative algebras and related Artin--Rees and continuity properties of Hochschild, cyclic, and Andr\'e--Quillen homology. To put these results in context, we first recall the Artin--Rees (aka.~Nilpotent Extension) theorem which can be found in the original works of Andr\'e \cite[Prop.~X.12]{Andre1974} and Quillen \cite[Thm.~6.15]{Quillen1970} on the cohomology of commutative rings: it states that if $A$ is a commutative, Noetherian ring and $I\subseteq A$ is an ideal, then the pro $A$-module $\{D_n(A/I^r|A)\}_r$, of Andr\'e--Quillen homologies for $A/I^r$ with respect to $A$, vanishes for all $n\ge0$. Assuming $A$ is a $k$-algebra for some commutative ring $k$, one can then apply the Jacobi--Zariski sequence to deduce a continuity isomorphism $\{D_n(A|k,A/I^r)\}_r\cong\{D_n(A/I^r|k)\}_r$. Some generalisations of these Artin--Rees and continuity results, in the special case of certain finite type algebras over characteristic zero fields, have been recently given by A.~Krishna \cite{Krishna2010}.

We present substantial generalisations of such results, for both Andr\'e--Quillen and Hochschild homology, in the axiomatic context of pro Tor-unital ideals; this is made possible by Theorem \ref{theorem1}, while Theorem \ref{theorem_comm_Noeth} ensures that we recover all earlier special cases. The following are our main such results in the case of Hochschild and cyclic homology:

\begin{theorem}[Artin--Rees \& continuity properties of $\HH$ and $\HC$ {[\S\ref{section_HH_and_HC}]}]\label{theorem_AR_intro}
Let $k\to A$ be a homomorphism of commutative rings, and $I$ a pro Tor-unital ideal of $A$. Then:
\begin{enumerate}
\item $\{\HH_n^A(A/I^r)\}=0$ for all $n>0$.
\item $\{\HC_n^A(A/I^r)\}=0$ for all odd $n>0$, and $is \cong\{A/I^r\}_r$ for all even $n\ge0$.
\item The canonical map $\{H_n^k(A,A/I^r)\}_r\to \{\HH_n^k(A/I^r)\}_r$ is an isomorphism for all $n\ge0$.
\end{enumerate}
\end{theorem}

From Theorem \ref{theorem_AR_intro}(iii) we obtain a strong form of the pro Hochschild--Kostant--Rosenberg theorem. A version of this result for finite type algebras over fields can be found in \cite[Thm.~3.2]{Cortinas2009}, but for recent applications to the formal deformation of algebraic cycles \cite{BlochEsnaultKerz2013, Morrow_Deformational_Hodge} the following strong version is required:

\begin{theorem}[Pro HKR theorem {[Thm.~\ref{theorem_pro_HKR}]}]\label{theorem_intro_HKR}
Let $k\to A$ be a geometrically regular morphism of commutative, Noetherian rings, and $I$ an ideal of $A$. Then the antisymmetrisation map \[\{\Omega_{A/I^r|k}^n\}_r\To\{\HH_n^k(A/I^r)\}_r\] is an isomorphism of pro $A$-modules for all $n\ge0$.
\end{theorem}

From Theorem \ref{theorem_intro_HKR} we then obtain a generalisation of the theorem of B.~Fe\u\i gin and B.~Tsygan \cite{FeiginTsygan1985} stating that $\HP^k_n(A)$ is isomorphic to crystalline cohomology (aka.~Hartshorne's algebraic de Rham cohomology) whenever $A$ is a finite type algebra over a characteristic zero field $k$. We extend this to general base rings $k$ and moreover eliminate the finite type hypothesis on $A$ (see Thm.~\ref{theorem_Feigin-Tsygan}).

Section \ref{section_AR_for_AQ} is devoted to Andr\'e--Quillen homology, where we in particular prove analogues of Theorem \ref{theorem_AR_intro} (see Corol.~\ref{corollary_AR_properties_in_AQ_homology} and Lem.~\ref{lemma_excision_1}). With such continuity properties established, we then show that pro Tor-unital ideals also satisfy pro excision in Andr\'e--Quillen homology, under a mild hypothesis of ``smallness'' which is always satisfied in the Noetherian case (see Def.~\ref{definition_small}):

\begin{theorem}[Pro excision for A.--Q.~homology {[Thm.~\ref{theorem_pro_excision}]}]\label{theorem3}
Let $k\to A\to B$ be homomorphisms of commutative rings, and $I$ a small, pro Tor-unital ideal of $A$ mapped isomorphically to an ideal of $B$. Then the following square of simplicial pro $A$-modules is homotopy cartesian:
\[\xymatrix{
\bb L_{A|k}^i\ar[r]\ar[d]&\bb L_{B|k}^i\ar[d]\\
\{\bb L_{A/I^r|k}^i\}_r\ar[r]&\{\bb L_{B/I^r|k}^i\}_r\\
}\]
where $\bb L^i_{-|-}$ denotes the $i^\sub{th}$ exterior power of the cotangent complex of a homomorphism of commutative rings.
\end{theorem}

\subsection*{Leitfaden}
Some of the main results of the paper are almost independent from others, so we offer an informal Leitfaden:
\begin{enumerate}\itemsep-1pt
\item The proof of Thm.~\ref{theorem1} is the contents of Section \ref{section_proof_of_1}.
\item Thm.~\ref{theorem_comm_Noeth} and Cor.~\ref{corollary2} follow from \S\ref{section_proof_of_1} and the start of \S\ref{section_Noetherian}.
\item Thms.~\ref{theorem4} and \ref{theorem5} follow from \S\ref{section_proof_of_1}, the introductory material of \S\ref{section_HH}, and \S\ref{subsection_pro_excision_for_HH}.
\item Theorem \ref{theorem_intro_CQ} follows from \S1, the introductory material of \S3, and \S3.2.
\item If one interprets pro Tor-unitality as the vanishing of $\{\Tor_n^A(A/I^r,A/I^r)\}_r$ for $n>0$, then Thm.~\ref{theorem_AR_intro}, Thm.~\ref{theorem_intro_HKR}, and our Fe\u\i gin--Tsygan theorem follow from the introductory material of \S\ref{section_HH}, and \S\ref{section_HH_and_HC}.
\item The Artin--Rees and continuity properties for Andr\'e--Quillen homology are self-contained in \S\ref{subsection_Artin_Rees_for_AQ}; Thm.~\ref{theorem3} additionally requires \S\ref{section_proof_of_1} and \S\ref{section_AQ_excision}.
\end{enumerate}

\subsection*{Notation, etc.}
Rings and algebras are associative and unital, unless explicitly stated to be non-unital, and ideals are two-sided. Given a commutative ring $k$ and a $k$-algebra $A$, an $A$-bimodule means an $A\otimes_kA^\sub{op}$-module. In Section~\ref{section_AR_for_AQ}, all rings are commutative. A discussion of pro abelian groups and pro modules may be found in Appendix \ref{appendix_pro}.

\subsection*{Acknowledgements}
Discussions with B.~Dundas during the preparation of an early version of this paper were very helpful, particularly his explanations of derived Hochschild and cyclic homology. The foundation of this work was done while I was supported by the Simons Foundation as a postdoctoral fellow at the University of Chicago, and later improvements were carried out at the University of Bonn while funded by the Hausdorff Center; I am grateful to both these sources for their support.

I also express my gratitude to the anonymous referee, who offered various improvements and encouraged me to use the techniques of the original version of the paper to prove the results now found in Sections \ref{section_converse} and \ref{subsection_HP}.

\section{The equivalence of pro unitality conditions}\label{section_proof_of_1}
In this section we define our pro unitality conditions of interest and establish the equivalences of Theorem \ref{theorem1} from the introduction. We claim no originality for the following notion, but must suggest a piece of nomenclature to be able to more clearly state our results:

\begin{definition}
Let $k$ be a commutative ring and $I$ a non-unital $k$-algebra. Then we say that $I$ is {\em pro Tor-unital over $k$} if and only if $\{\Tor_n^{k\ltimes I^r}(k,k)\}_r=0$ for all $n>0$. If $k=\bb Z$, we say simply that $I$ is {\em pro Tor-unital}.
\end{definition}

The main aim of this section is to prove the following equivalences, which we will then adopt as alternative definitions of pro Tor-unitality:

\begin{theorem}\label{theorem1a}
Let $k$ be a commutative ring, $A\to B$ a homomorphism of $k$-algebras, and $I$ an ideal of $A$ mapped isomorphically to an ideal of $B$. Then the following pro Tor vanishing conditions are equivalent:
\begin{enumerate}
\item $\{\Tor_n^A(A/I^r,A/I^r)\}_r=0$ for all $n>0$.
\item $\{\Tor_n^B(B/I^r,B/I^r)\}_r=0$ for all $n>0$.
\item $I$ is pro Tor-unital over $k$, i.e., $\{\Tor_n^{k\ltimes I^r}(k,k)\}_r=0$ for all $n>0$.
\item $I$ is pro Tor-unital, i.e.,  $\{\Tor_n^{\bb Z\ltimes I^r}(\bb Z,\bb Z)\}_r=0$ for all $n>0$.
\end{enumerate}
\end{theorem}

Before we prove Theorem \ref{theorem1a}, we present various situations in which it can be applied:

\begin{example}[Commutative, Noetherian rings]
Let $I$ be an ideal of a commutative, Noetherian ring $A$. Then condition (i) of Theorem \ref{theorem1a} is true (see Thm.~\ref{theorem_pro_Tor_unital_commutative}) and hence $I$ is pro Tor-unital.
\end{example}

\begin{example}[Quasi-regularity]
Let $I$ be a {\em quasi-regular} ideal of a commutative ring $A$ in Quillen's sense \cite[Def.~6.10]{Quillen1970}, i.e., $I/I^2$ is flat as an $A/I$-module and the canonical map $\bigwedge_{A/I}^nI/I^2\to\Tor_n^A(A/I,A/I)$ is an isomorphism for all $n\ge0$. For example, it suffices that $I$ be generated by a regular sequence. We will now show that condition (i) of Theorem \ref{theorem1a} is satisfied, whence $I$ is pro Tor-unital.

It was proved by Quillen that the canonical map \[\Tor_n^A(A/I^s,A/I^{r+s})\To\Tor_n^A(A/I^s,A/I^r)\] is zero for all $n,r,s>0$. (Since this precise assertion is not stated by Quillen, we now explain the missing details of the proof. Firstly, if $s=1$ then this is precisely the statement ``$B_m$'' in the proof of \cite[Prop.~8.5]{Quillen1968}. Since $I^s/I^{s+1}$ is flat over $A/I$ \cite[Prop.~8.5]{Quillen1968}, it follows that the map $\Tor_n^A(I^s/I^{s+1},A/I^{r+1})\to\Tor_n^A(I^s/I^{s+1},A/I^r)$ is zero for all $n,r,s>0$; now a straightforward induction on $s$ completes the proof.) Hence the composition \[\Tor_n^A(A/I^{2r},A/I^{2r})\To\Tor_n^A(A/I^r,A/I^{2r})\To\Tor_n^A(A/I^r,A/I^r)\] is zero for all $n,r>0$, as required.
\end{example}

\begin{example}[Restriction of base ring]
Suppose that $k'\to k$ is a homomorphism of commutative rings and that $I$ is a non-unital $k$-algebra. Then Theorem \ref{theorem1a} evidently implies that the following two conditions are equivalent:
\begin{enumerate}\itemsep0pt
\item[(iii)] $I$ is pro Tor-unital over $k$, i.e., $\{\Tor_n^{k\ltimes I^r}(k,k)\}_r=0$ for all $n>0$.
\item[(iii')] $I$ is pro Tor-unital over $k'$, i.e., $\{\Tor_n^{k'\ltimes I^r}(k',k')\}_r=0$ for all $n>0$.
\end{enumerate}
This provides a strengthening and purely algebraic proof of a result of Geisser and Hesselholt \cite[Prop.~3.6]{GeisserHesselholt2011}, as promised in the author's earlier work \cite[Lemma 1.12]{Morrow_birelative_dim1}.
\end{example}

\begin{example}[H-unitality]\label{example_pro_H}
Let $k$ be a commutative ring and $I$ a non-unital $k$-algebra. Assuming that $I$ is flat over $k$, we may calculate $\Tor_n^{k\ltimes I}(k,I)=\Tor_{n+1}^{k\ltimes I}(k,k)$ as the $n^\sub{th}$ homology of the bar complex \[B_\blob^k(I):=\quad 0\longleftarrow I \stackrel{b'}{\longleftarrow} I\otimes_kI\stackrel{b'}{\longleftarrow} I\otimes_kI\otimes_kI\stackrel{b'}{\longleftarrow}\cdots\] (Here $I$ sits in degree $0$ of the complex, and the boundary maps $b'$ are given by the alternating sum of $x_0\otimes\cdots\otimes x_n\mapsto x_0\otimes\cdots \otimes x_ix_{i+1}\otimes\cdots\otimes x_n$ for $i=0,\dots,n-1$.) 

In particular, if $k$ is a field so that $I^r$ is flat over $k$ for all $r\ge1$, then we see that $I$ is pro Tor-unital over $k$ if and only $\{H_n(B_\blob^k(I^r))\}_r=0$ for all $n\ge0$; this latter condition, known as {\em H-unitality} of the pro $k$-algebra $\{I^r\}_r$, has appeared previously, e.g., \cite[\S4.2]{Cortinas2006}.
\end{example}

\begin{example}[Quasi-unitality, ideals of free algebras]\label{example_quasi_unital}
Suppose $k$ is a field and that $I$ is a non-unital $k$-algebra. Then $I$ is said to be {\em quasi-unital} if and only if there exists a $k$-linear map $\al:I^2\to I\otimes_kI$ with the following properties:
\begin{itemize}\itemsep0pt
\item $\al(xy)=x\al(y)$ for $x\in I$, $y\in I^2$.
\item $\mu\circ\al=\op{id}$, where $\mu:I\otimes_kI\to I^2$ denotes multiplication.
\end{itemize}
When this is true, we may define (using the bar construction of Example \ref{example_pro_H}) \[\sigma:\underbrace{I^2\otimes_k \cdots\otimes_k I^2}_{n\sub{ times}}\to\underbrace{I\otimes_k \cdots\otimes_k I}_{n+1\sub{ times}},\quad x_0\otimes\cdots\otimes x_n\mapsto \al(x_0)\otimes x_1\otimes\cdots\otimes x_n,\] which is easily seen to have the property that $b'\sigma+\sigma b'=i$, where $i:B_\blob^k(I^2)\to B_\blob^k(I)$ is the canonical map; hence $i$ induces zero on homology, i.e, the canonical map $\Tor_n^{k\ltimes I^2}(k,k)\to \Tor_n^{k\ltimes I}(k,k)$ is zero for all $n>0$.

When Cuntz and Quillen proved excision in periodic cyclic homology over fields of characteristic zero, they considered non-unital $k$-algebras $I$ such that $I^{2^r}$ is quasi-unital for all $r\ge0$. For example, this is true if $I$ is an ideal of a free $k$-algebra by \cite[\S4]{CuntzQuillen1997}, and so we have just explained that such ideals are pro Tor-unital over $k$. A generalisation of this assertion to the case when $k$ is not necessarily a field will be given in Proposition \ref{proposition_pro_Tor_unital_in_free}. Abstract generalisations to monoidal categories may be found in~\cite{CortinasValqui2003}.

\end{example}

Now we turn to the proof of Theorem \ref{theorem1a}; it will be deduced from a more general result concerning pro rings, for which we need first to mention some terminology. The reader less familiar with pro objects may also wish to consult Appendix \ref{appendix_pro}.

A pro ring $R_\infty$ is simply an object of the category $\op{Pro}Rings$, i.e.~an inverse system $\cdots\to R_2\to R_1$ of (always unital and associative) rings. An {\em ideal} $I_\infty\subseteq R_\infty$ is an inverse system $\cdots\to I_2\to I_1$ where each $I_r$ is an ideal of $R_r$, and where the transition maps $R_{r+1}\to R_r$ restricts to the transition maps $I_{r+1}\to I_r$. The notation $R_\infty/I_\infty$ denotes the pro ring $\{R_r/I_r\}_r$.

A {\em (strict) left $R_\infty$-module} $M_\infty$ is by definition an inverse system $\cdots\to M_2\to M_1$, where $M_r$ is a left $R_r$-module, and the transition maps are compatible in the obvious sense. A {\em (strict) right $R_\infty$-module} is defined in the analogous fashion. All our pro modules will be strict, so we will not mention this assumption again.

The following is the key proposition:

\begin{proposition}\label{proposition_pro_H}
Let $R_\infty\to S_\infty$ be a strict map of pro rings, and suppose that $J_\infty$ is an ideal of $R_\infty$ such that each map $R_r\to S_r$ carries $J_r$ isomorphically to an ideal of $S_r$. Then the following pro Tor vanishing conditions are equivalent:
\begin{enumerate}
\item $\{\Tor_n^{R_r}(R_r/J_r,R_r/J_r)\}_r=0$ for all $n>0$.
\item $\{\Tor_n^{S_r}(S_r/J_r,S_r/J_r)\}_r=0$ for all $n>0$.
\end{enumerate}
\end{proposition}

Theorem \ref{theorem1a} follows from Proposition \ref{proposition_pro_H}:

\begin{proof}[\bf Proof of Theorem \ref{theorem1a} from Proposition \ref{proposition_pro_H}]
Let $k$ be a commutative ring, $A\to B$ a homomorphism of $k$-algebras, and $I$ an ideal of $A$ mapped isomorphically to an ideal of $B$. We must show that conditions (i)--(iv) of Theorem \ref{theorem1} are equivalent; we do this by applying the previous proposition to a variety of pro rings, always with ideal $J_\infty=\{I^r\}_r$:

(i)$\Leftrightarrow$(ii): $R_\infty=A$ and $S_\infty=B$. (iii)$\Leftrightarrow$(i): $R_\infty=\{k\ltimes I^r\}_r$ and $S_\infty=A$. (iv)$\Leftrightarrow$(iii): $R_\infty=\{\bb Z\ltimes I^r\}_r$ and $S_\infty=\{k\ltimes I^r\}_r$. 
\end{proof}

\begin{remark}\label{remark_non_pro}
Given that Proposition \ref{proposition_pro_H} has now been shown to be the key result, it may be instructive first to establish its non-pro version to illustrate the idea of proof.

Let $R\to S$ be a homomorphism of rings, and suppose that $J$ is an ideal of $R$ carried isomorphically to an ideal of $S$. Then we claim that the following are equivalent:
\begin{enumerate}\itemsep0pt
\item[(a)] $\Tor_n^R(R/J,R/J)=0$ for all $n>0$.
\item[(b)] $\Tor_n^S(S/J,S/J)=0$ for all $n>0$.
\end{enumerate}
That is, we claim that $R/J\dotimes_RR/J\simeq R/J$ if and only if $S/J\dotimes_SS/J\simeq S/J$, where we work in the model category of simplicial rings. The proof is as follows:

(a)$\Rightarrow$(b): Assume that $R/J\dotimes_RR/J\simeq R/J$. Then for any right $R/J$-module $M$, we have \[M\dotimes_RR/J\simeq M\dotimes_{R/J}R/J\dotimes_RR/J\stackrel{(*)}{\simeq} M\dotimes_{R/J}R/J\simeq M,\] where ($*$) follows from our assumption. In particular, \[S/J\dotimes_RR/J\simeq S/J\tag{\dag}.\]
Next, our assumption implies that the canonical map $R\dotimes_RR/J\to R/J\dotimes_RR/J$ is a weak equivalence, hence that $J\dotimes_RR/J\simeq 0$. It follows that the canonical map $S\dotimes_RR/J\to S/J\dotimes_RR/J$ is a weak equivalence and hence, by (\dag), that $S\dotimes_R R/J\simeq S/J$. Therefore \[S/J\dotimes_SS/J\simeq S/J\dotimes_SS\dotimes_RR/J\simeq S/J\dotimes_RR/J\simeq S/J,\] where the final equivalence follows from (\dag).

(b)$\Rightarrow$(a): Assume that $S/J\dotimes_SS/J\simeq S/J$. For any simplicial left $S/J$-module $N$, the same argument as in the start of the previous implication shows that $S/J\dotimes_SN\simeq N$; hence the natural map $S\dotimes_SN\to S/J\dotimes_SN$ is a weak equivalence, and therefore $J\dotimes_SN\simeq0$. Applying this with $N=S\dotimes_RR/J$, we deduce that $J\dotimes_RR/J\simeq 0$. Hence the canonical map $R\dotimes_RR/J\to R/J\dotimes_RR/J$ is a weak equivalence, i.e.,~$R/J\dotimes_RR/J\simeq R/J$.

This completes the proof of the claim, i.e., of Proposition \ref{proposition_pro_H} in the non-pro setting.
\end{remark}

The reader who is comfortable with model categories of pro simplicial abelian groups and pro simplicial rings (e.g., using the machinery of \cite{Isaksen2001}) may presumably verbatim repeat the argument given in Remark \ref{remark_non_pro} in order to prove Proposition \ref{proposition_pro_H}. We will not present such a proof, but instead will work directly with spectral sequences of pro abelian groups instead of derived tensor products. More precisely, the following standard change of rings spectral sequences for a homomorphism of rings will be used frequently in the proof of the proposition:

\begin{remark}\label{remark_base_change_SS}
Let $R\to S$ be a homomorphism of rings. If $M$ is a right $S$-module and $N$ is a left $R$-module, then there is a first quadrant spectral sequence \[E^2_{pq}=\Tor_p^S(M,\Tor_q^R(S,N))\Longrightarrow\Tor_{p+q}^R(M,N).\] If $N$ is a left $S$-module and $M$ is a right $R$-module, then there is a first quadrant spectral sequence \[^\prime E^2_{pq}=\Tor_p^S(\Tor_q^R(M,S), N)\Longrightarrow\Tor_{p+q}^R(M,N).\]
\end{remark}

We now begin our proof of Proposition \ref{proposition_pro_H} with the following lemma:

\begin{lemma}\label{lemma_pro_result_of_Tor_vanishing}
Let $J_\infty$ be an ideal of a pro ring $R_\infty$, and assume that the pro abelian groups $\{\Tor_n^{R_r}(R_r/J_r,R_r/J_r)\}_r$ vanish for all $n>0$. Then, for any right (resp., left) $R_\infty/J_\infty$-module $M_\infty$ (resp., $N_\infty$), the canonical map \[\{\Tor_n^{R_r}(M_r,N_r)\}_r\To\{\Tor_n^{R_r/J^r}(M_r,N_r)\}_r\] is an isomorphism of pro abelian groups for all $n\ge 0$.

In particular, $\{\Tor_n^{R_r}(M_r,R_r/J_r)\}_r$ and $\{\Tor_n^{R_r}(R_r/J_r,N_r)\}_r$ vanish for all $n>0$.
\end{lemma}
\begin{proof}
For each $r\ge1$ Remark \ref{remark_base_change_SS} gives us a spectral sequence of abelian groups
\[E^2_{pq}(r)=\Tor_p^{R_r/J_r}(M_r,\Tor_q^{R_r}(R_r/J_r,R_r/J_r))\Longrightarrow\Tor_{p+q}^{R_r}(M_r,R_r/J_r).\] Letting $r\to\infty$ gives a first quadrant spectral sequence of pro abelian groups \[E^2_{pq}(\infty)=\{\Tor_p^{R_r/J_r}(M_r,\Tor_q^{R_r}(R_r/J_r,R_r/J_r))\}_r\Longrightarrow\{\Tor_{p+q}^{R_r}(M_r,R_r/J_r)\}_r.\] But our vanishing assumption implies that $E_{pq}^2(\infty)=0$ unless $q=0$, and so the spectral sequence degenerates to edge map isomorphisms \[\{\Tor_n^{R_r}(M_r,R_r/J_r)\}_r\Isoto \{\Tor_n^{R_r/J_r}(M_r,R_r/J_r)\}_r.\] Since the right side vanishes for $n>0$, we have proved \[\{\Tor_n^{R_r}(M_r,R_r/J_r)\}_r\cong\begin{cases}\{M_r\}_r&n=0,\\0&n>0.\end{cases}\]

Applying the same argument to the second spectral sequence of Remark \ref{remark_base_change_SS} obtains a spectral sequence of pro abelian groups \[' E^2_{pq}(\infty)=\{\Tor_p^{R_r/J_r}(\Tor_q^{R_r}(M_r,R_r/J_r),N_r)\}_r\Longrightarrow\{\Tor_{p+q}^{R_r}(M_r,N_r)\}_r.\] But by what we have just proved, $'E_{pq}^2(\infty)$ vanishes unless $q=0$, so we obtain the desired edge map isomorphisms $\{\Tor_n^{R_r}(M_r,N_r)\}_r\isoto\{\Tor_n^{R_r/J_r}(M_r,N_r)\}_r$.

The ``in particular'' claims immediately follow.
\end{proof}

As in the statement of Proposition \ref{proposition_pro_H}, let $R_\infty\to S_\infty$ be a strict map of pro rings, and suppose that $J_\infty$ is an ideal of $R_\infty$ such that each map $R_r\to S_r$ carries $J_r$ isomorphically to an ideal of $S_r$. Then, for each $r\ge1$, the square of $R_r$-modules
\[\xymatrix{
R_r\ar[d]\ar[r] & S_r\ar[d]\\
R_r/J_r\ar[r]& S_r/J_r
}\]
is bicartesian, i.e., both cartesian and cocartesian. In other words, the sequence \[0\to R_r\to R_r/J_r\oplus S_r\to S_r/J_r\to 0\] is a short exact sequence of $R_r$-modules, to which we can apply derived functors to get long exact sequences. In particular, applying $\Tor_*^{R_r}(-,R_r/J_r)$ yields a long exact sequence \[\hspace{-0.5cm}\cdots\to\Tor_n^{R_r}(R_r,R_r/J_r)\to \Tor_n^{R_r}(R_r/J_r,R_r/J_r)\oplus \Tor_n^{R_r}(S_r,R_r/J_r)\to \Tor_n^{R_r}(S_r/J_r,R_r/J_r)\to\cdots\] But the left-most group vanishes for $n>0$, giving \[\Tor_n^{R_r}(R_r/J_r,R_r/J_r)\oplus \Tor_n^{R_r}(S_r,R_r/J_r)\Isoto \Tor_n^{R_r}(S_r/J_r,R_r/J_r)\tag{\dag}\] for $n>0$ (the reader will easily provide the small detail needed when $n=1$).

Now we finish the proof of Proposition \ref{proposition_pro_H}:
\begin{proof}[Proof of Proposition \ref{proposition_pro_H}]
Let $R_\infty\to S_\infty$ be a strict map of pro rings, and suppose that $J_\infty$ is an ideal of $R_\infty$ such that each map $R_r\to S_r$ carries $J_r$ isomorphically to an ideal of $S_r$. We must prove the equivalence of the conditions given in the statement of Proposition \ref{proposition_pro_H}.

(i)$\Rightarrow$(ii): Assume that $\{\Tor_n^{R_r}(R_r/J_r,R_r/J_r)\}_r=0$ for $n>0$. When we take the limit over $r$ in (\dag), the left pro abelian group vanishes by assumption, while the right pro abelian group vanishes by Lemma \ref{lemma_pro_result_of_Tor_vanishing}; therefore $\{\Tor_n^{R_r}(S_r,R_r/J_r)\}_r=0$ for $n>0$. 

By Remark \ref{remark_base_change_SS} we have spectral sequences \[E^2_{pq}(r)=\Tor_p^{S_r}(S_r/J_r,\Tor_q^{R_r}(S_r,R_r/J_r))\Longrightarrow\Tor_{p+q}^{R_r}(S_r/J_r,R_r/J_r).\] We have just proved that the internal Tor groups on the $E^2$ page vanish when $r\to\infty$, and Lemma \ref{lemma_pro_result_of_Tor_vanishing} implies again that the abutment vanishes in degrees $>0$ when $r\to\infty$. Thus the limit of the spectral sequences collapses to the statement that $\{\Tor_n^{S_r}(S_r/J_r,S_r/J_r)\}_r=0$ for $n>0$, as desired.

(ii)$\Rightarrow$(i): Now assume instead that $\{\Tor_n^{S_r}(S_r/J_r,S_r/J_r)\}_r=0$ for all $n>0$. If $N_\infty$ is a left $S_\infty/J_\infty$-module then Lemma \ref{lemma_pro_result_of_Tor_vanishing} implies that $\{\Tor_n^{S_r}(S_r/J_r,N_r)\}_r=0$ for $n>0$; then the long exact Tor sequence for $0\to J_r\to S_r\to S_r/J_r\to 0$ implies that $\{\Tor_n^{S_r}(J_r,N_r)\}_r=0$ for all $n\ge0$.

Next, taking the limit of another change of rings spectral sequences gives \[E^2_{pq}(\infty)=\{\Tor_p^{S_r}(J_r,\Tor_q^{R_r}(S_r,R_r/J_r))\}_r\Longrightarrow\{\Tor_{p+q}^{R_r}(J_r,R_r/J_r)\}_r.\] But from (\dag) we see that $N_r:=\Tor_q^{R_r}(S_r,R_r/J_r)$ is, as a left $S_r$-module, actually a $S_r/J_r$-module for $q>0$; so $N_\infty:=\{\Tor_q^{R_r}(S_r,R_r/J_r)\}_r$ is a left $S_\infty/J_\infty$-module and therefore we have just shown that $E_{pq}^2(\infty)=0$ for $p\ge0$. Thus the spectral sequence degenerates to edge map isomorphisms \[\{\Tor_n^{R_r}(J_r,R_r/J_r)\}_r\Isoto\{\Tor_n^{S_r}(J_r,S_r/J_r)\}_r\] for $n\ge0$. But the right pro Tor group vanishes for all $n\ge0$, by applying the previous paragraph to the left $S_\infty/J_\infty$-module $N_\infty=S_\infty/J_\infty$.

Finally, apply $\Tor_*^{R_r}(-,R_r/J_r)$ to the short exact sequence $0\to J_r\to R_r\to R_r/J_r\to0$ and let $r\to\infty$ to get $\{\Tor_n^{R_r}(R_r/J_r,R_r/J_r)\}_r=0$ for $n>0$, completing the proof.

This completes the proof of Proposition \ref{proposition_pro_H}, hence also of Theorem \ref{theorem1a}.
\end{proof}

We finish this section by noting the following useful corollary of our calculations, which will be needed several times later in the paper:

\begin{corollary}\label{corollary_most_useful_pro_vanishing}
Let $I$ be a pro Tor-unital ideal of a ring $A$, and let $M$ (resp., $N$) be a right (resp., left) $A$-module. Then the pro abelian groups \[\{\Tor_n^A(M/MI^r,A/I^r)\}_r\quad\mbox{and}\quad\{\Tor_n^A(A/I^r,N/I^rN)\}_r\] vanish for all $n>0$.
\end{corollary}
\begin{proof}
By assumption and Theorem \ref{theorem1a}, $\{\Tor_n^A(A/I^r,A/I^r)\}_r=0$ for $n>0$; so the corollary follows by applying the ``In particular'' claim of Lemma \ref{lemma_pro_result_of_Tor_vanishing} to $R_\infty=A$, $J_\infty=\{I^r\}_r$, and $M_\infty=\{M/MI^r\}_r$ (resp., $N_\infty=\{N/I^rN\}_r$).
\end{proof}

\section{The case of commutative, Noetherian rings}\label{section_Noetherian}
In this section we apply Theorem \ref{theorem1a} to ideals of commutative, Noetherian rings in order to prove that such ideals satisfy pro excision in algebraic $K$-theory. We first recall a result of Andr\'e:

\begin{lemma}[{Andr\'e's Artin--Rees property \cite[Lem.~X.11]{Andre1974}}]\label{lemma_Andre}
Suppose that $I$ is an ideal of a commutative, Noetherian ring $A$. Then:
\begin{enumerate}
\item If $M$ is a finitely generated $A$-module, then $\{\Tor_n^A(M,A/I^r)\}_r=0$ for all $n>0$.
\item $\{\Tor_n^A(A/I^r,A/I^r)\}_r=0$ for all $n>0$.
\end{enumerate}
\end{lemma}
\begin{proof}
(i): Let $P_\bullet\to M$ be a resolution of $M$ by finitely generated, projective $A$-modules. Then, for any $n,r\ge1$, one applies the classical Artin--Rees lemma to the $A$-modules $d(P_n)\subseteq P_{n-1}$ to deduce that there exists $s\ge r$ such that $I^sP_{n-1}\cap d(P_n)\subseteq d(I^rP_n)$. Taking the preimage under the differential $d$, one obtains $d^{-1}(I^sP_{n-1})\subseteq I^rP_n+d(P_{n+1})$, i.e., the map $\Tor_n^A(M,A/I^s)\to\Tor_n^A(M,A/I^r)$ is zero, as required.

(ii): Applying (i) to the $A$-module $M=A/I^r$ we find $s\ge r$ such that the map $\Tor_n^A(A/I^r, A/I^s)\to \Tor_n^A(A/I^r, A/I^r)$ is zero. So the map $\Tor_n^A(A/I^s, A/I^s)\to \Tor_n^A(A/I^r, A/I^r)$ is certainly zero, as required.
\end{proof}

\begin{remark}
Lemma \ref{lemma_Andre} and its proof are also valid for ideals of non-commutative, Noetherian rings which satisfy the Artin--Rees property. See \cite[Chap.~4, \S2]{McConnell2001} for precise definitions and examples.
\end{remark}

We may now prove Theorem \ref{theorem_comm_Noeth} and Corollary \ref{corollary2} from the Introduction:

\begin{theorem}\label{theorem_pro_Tor_unital_commutative}
Let $I$ be an ideal of a commutative, Noetherian ring. Then $I$ is pro Tor-unital.
\end{theorem}
\begin{proof}
This follows from Lemma \ref{lemma_Andre}(ii) and the equivalences of Theorem \ref{theorem1}.
\end{proof}

\begin{corollary}\label{corollary2a}
Ideals of commutative, Noetherian rings satisfy pro excision in algebraic $K$-theory. In particular, if $A\to B$ is a homomorphism of commutative, Noetherian rings, and $I$ is an ideal of $A$ mapped isomorphically to an ideal of $B$, then the map of pro abelian groups \[\{K_n(A,I^r)\}_r\To \{K_n(B,I^r)\}_r\] is an isomorphism for all $n\in\bb Z$.
\end{corollary}
\begin{proof}
Let $I$ be an ideal of a commutative Noetherian ring. According to Theorem \ref{theorem_pro_Tor_unital_commutative}, $I$ is pro Tor-unital. Moreover, according to Geisser--Hesselholt's pro version of the Suslin--Wodzicki criterion \cite[Thm.~3.1]{GeisserHesselholt2011} (see \cite[Thm.~1.1]{GeisserHesselholt2006} for the proof), this is sufficient to imply that $I$ satisfy pro excision in algebraic $K$-theory.
\end{proof}

We finish this section on pro excision for commutative, Noetherian rings by giving various examples of situations in which Corollary \ref{corollary2a} in can be used, and an interpretation in terms of formal schemes. Although we focus on $K$-theory, we stress that these examples apply verbatim to Hochschild and cyclic homology, their topological counterparts, and Andr\'e--Quillen homology, since we will establish the analogue of Corollary \ref{corollary2a} for these homology theories in Theorems \ref{theorem_pro_excision_for_HH_HC}, \ref{theorem_pro_THH}, and \ref{theorem_pro_excision} respectively.

\begin{example}[Conductor ideals]
Let $A$ be a commutative, Noetherian ring which is reduced; let $I$ be any ideal of the normalisation $\tilde A$ which is contained inside $A$, such as the conductor ideal $\Ann_A(\tilde A/A)$. Then Corollary \ref{corollary2a} implies that \[\{K_n(A,I^r)\}_r\cong\{K_n(\tilde A,I^r)\}_r\] for all $n\in\bb Z$. This has been previously proved by much longer arguments in the following two special cases: $A$ is finitely generated over a characteristic zero field and $\tilde A$ is assumed to be smooth, by Krishna \cite[Thm.~1.1]{Krishna2010}; $A$ is one-dimensional, by the author \cite{Morrow_birelative_dim1}.
\end{example}

\begin{example}[Pro Mayer--Vietoris for closed covers]\label{example_closed_covers}
Let $A$ be a commutative, Noetherian ring, and $I,J\subseteq A$ ideals such that $I\cap J=0$. For any $e\ge1$ we may apply Corollary \ref{corollary2a} to the ring homomorphism $A\to A/I^e$ and ideal $J\subseteq A$ to deduce that the canonical map $\{K_n(A,J^r)\}_r\to\{K_n(A/I^e,(I^e+J^r)/I^e)\}_r$ is an isomorphism for all $n\in\bb Z$. Taking the ``diagonal over $e,r$'' it easily follows that the canonical map $\{K_n(A,J^r)\}_r\to\{K_n(A/I^r,(I^r+J^r)/I^r)\}_r$ is an isomorphism for all $n\in\bb Z$; moreover, we may replace $I^r+J^r$ by $(I+J)^r$ since these two chains of ideals are intertwined. More symmetrically, we have proved that the square of pro spectra
\[\xymatrix{
K(A)\ar[r]\ar[d] & \{K(A/J^r)\}_r\ar[d]\\
\{K(A/I^r)\}_r\ar[r]&\{K(A/(I+J)^r)\}_r
}\]
is homotopy cartesian.
\end{example}

\begin{example}\label{example_formal_closed_covers}
Let $A$ be a commutative, Noetherian ring, and $I,J\subseteq A$ ideals. For any $e\ge1$ we may apply Example \ref{example_closed_covers} to the ring $A/(I^e\cap J^e)$ and ideals $I^e/(I^e\cap J^e)$, $J^e/(I^e\cap J^e)$ to deduce that the canonical map \[\left\{K_n\left(\frac{A}{I^e\cap J^e},\frac{J^{er}+(I^e\cap J^e)}{I^e\cap J^e}\right)\right\}_r\to \left\{K_n\left(\frac{A}{I^{er}+(I^e\cap J^e)},\frac{I^{er}+J^{er}+(I^e\cap J^e)}{I^{er}+(I^e\cap J^e)}\right)\right\}_r\] is an isomorphism for all $n\in\bb Z$. Taking the diagonal over $e,r$, it follows that  \[\{K_n(A/(I^r\cap J^r),J^r/(I^r\cap J^r)\}_r\isoto\{K_n(A/I^r,(I^r+J^r)/I^r)\}_r.\] Using intertwinedness of the chains of ideals $I^r\cap J^r$ and $I^rJ^r$, and of $I^r+J^r$ and $(I+J)^r$, we may rewrite this as \[\{K_n(A/I^rJ^r,J^r/I^rJ^r)\}_r\isoto\{K_n(A/I^r,(I+J)^r/I^r)\}_r.\] In other words, the square of pro spectra 
\[\xymatrix{
\{K(A/I^rJ^r)\}_r\ar[d]\ar[r]&\{K(A/J^r)\}_r\ar[d]\\
\{K(A/I^r)\}_r\ar[r]&\{K(A/(I+J)^r)\}_r
}\]
is homotopy cartesian.
\end{example}

Example \ref{example_formal_closed_covers} may be interpreted as a statement concerning the $K$-theory of formal schemes. If $\frak X$ is a formal scheme, recall that an {\em ideal of definition} for $\frak X$ is an ideal sheaf $\cal J\subseteq\roi_\frak X$ such that $(\frak X,\roi_\frak X/\cal J)$ is a Noetherian scheme, which we will denote by $\frak X/\cal J$; any power of an ideal of definition is again an ideal of definition, and any ideal of definition contains a power of any other ideal of definition \cite[Prop.~9.5]{Hartshorne1977}. So, fixing any ideal of definition $\cal J$, the spectrum \[K(\frak X):=\holim_r K(\frak X/\cal J^r)\] depends only on $\frak X$ and not on the choice of $\cal J$; we denote the homotopy groups of $K(\frak X)$ by $K_n(\frak X)$ and call them them the $K$-groups of the formal scheme $\frak X$. (We could alternatively consider the pro spectrum $\{K(\frak X/\cal J^r)\}_r$, but we have opted to take the homotopy limit.)

Corollary \ref{corollary2a} proves that this $K$-theory of formal schemes has the Mayer--Vietoris property with respect to closed covers:

\begin{theorem}\label{theorem_formal_schemes}
Let $\frak X$ be a Noetherian formal scheme, and $\frak Y,\frak Z\into \frak X$ closed formal subschemes of $\frak X$ such that $\frak X=\frak Y\cup\frak Z$ (set-theoretically). Then the square of spectra
\[\xymatrix{
K(\frak X)\ar[r]\ar[d]&K(\frak Z)\ar[d]\\
K(\frak Y)\ar[r]&K(\frak Y\cap\frak Z)
}\]
is homotopy cartesian.
\end{theorem}
\begin{proof}
Since $K$-theory satisfies Zariski descent, it is easy to see that the $K$-theory of formal schemes has the Mayer--Vietoris property with respect to an open cover. This reduces the assertion of the theorem to the affine case, where it follows by taking homotopy limits in the main conclusion of Example \ref{example_formal_closed_covers}.
\end{proof}

\section{Pro excision and continuity in cyclic homology}\label{section_HH}
This section has four distinct goals:
\begin{itemize}\itemsep0pt
\item[\S\ref{subsection_pro_excision_for_HH}.] We prove that pro Tor-unital ideals satisfy pro excision in derived Hochschild and cyclic homology, and in topological Hochschild and cyclic homology.
\item[\S\ref{section_converse}.] We prove the converse statement, namely that pro excision forces ideals to be pro Tor-unital.
\item[\S\ref{subsection_HP}.] We extend the Cuntz--Quillen theorem on excision in periodic cyclic homology from base fields of characteristic zero to arbitrary commutative $\bb Q$-algebras.
\item[\S\ref{section_HH_and_HC}.] We restrict to commutative algebras, presenting Artin--Rees and continuity properties of Hochschild homology, including strong forms of the Fe\u\i gin--Tsygan and pro Hochschild--Kostant--Rosenberg theorems.
\end{itemize}

We begin by explaining the derived Hochschild and cyclic homology which will concern us. We assume that the reader is familiar with the basic language and tools of homotopical algebra due to Quillen \cite{Quillen1967}; we will occasionally implicitly identify simplicial objects with chain complexes via the Dold--Kan correspondence, but this should not cause any confusion.

Let $k$ be a commutative ring and $A$ a $k$-algebra. Given an $A$-bimodule $M$, we let  $H_*^{\sub{naive},k}(A,M)$  denote the ``usual'' Hochschild homology of $A$ as a $k$-algebra with coefficients in $M$; it is the homotopy of the Hochschild simplicial $k$-module $C_\bullet^k(A,M)$. In particular, $\HH_*^{\sub{naive},k}(A):=H_*^{\sub{naive},k}(A,A)$ denotes the usual Hochschild homology of $A$ as a $k$-algebra. However, we will work throughout with the derived version of Hochschild homology, for which we use the notation $H_*^k(A,M)$. That is, letting $P_\bullet\to A$ be a simplicial resolution of $A$ by free $k$-algebras, $H_*^k(A,M)$ is defined to be the homotopy (of the diagonal) of the bisimplicial $k$-module $C_\blob^k(P_\blob,M)$ given by \[p,q\mapsto C_q^k(P_p,M)=M\otimes_kP_p^{\otimes_k q}.\] In the special case $A=M$ we write $C_\blob^k(A):=C_\blob^k(A,A)$ and $\HH_*^k(A):=H_*^k(A,A)$.

\begin{remark}[Shukla homology]
Derived Hochschild homology is also called {\em Shukla homology} after \cite{Shukla1961}. Indeed, the Shukla homology $\op{Shuk}_*^k(A,M)$ of a $k$-algebra $A$ with coefficients in an $A$-bimodule $M$ is usually defined as the homology of the totalisation of the simplicial complex given by $q\mapsto M\otimes_k D_\blob^{\otimes_kq}$, where $D_\blob\to A$ is a projective, differential graded $k$-algebra resolution of $A$ and $\otimes_k$ denotes a graded tensor product; letting $D_\blob=CP_\blob$ be the unnormalised complex (equipped with the shuffle product as usual) of a simplicial resolution $P_\blob\to A$, and appealing to the Eilenberg--Zilber weak equivalence $M\otimes_k (CP_\blob)^{\otimes_kq}\simeq M\otimes_kP_\blob^{\otimes_kq}$, reveals that $\op{Shuk}_*^k(A,M)\cong H_*^k(A,M)$.
\end{remark}

\begin{remark}\label{remark_enveloping_alg}
A standard argument with bar complexes, as in \cite[Prop.~1.1.13]{Loday1992}, shows that $H_*^k(A,M)$ is given by the homotopy of the bisimplicial $k$-module $p\mapsto M\dotimes_{P_p\otimes_kP_p^\sub{\tiny op}}A$. More concisely, this means that $H_*^k(A,M)$ is given by the homotopy of the simplicial $k$-module \[M\dotimes_{A\dotimes_kA^\sub{\tiny op}}A,\] where we use derived tensor products over simplicial rings \cite[\S II.6]{Quillen1967}.
\end{remark}

Next we discuss cyclic homology. Firstly, $\HC_*^{\sub{naive},k}(A)$ denotes the usual cyclic homology of the $k$-algebra $A$; i.e., the homology of $\CC_{\blob}^k(A)$, which denotes the totalisation of Tsygan's bicomplex  $\CC_{\blob\blob}^k(A)$. Just as for Hochschild homology we prefer to denote by $\HC_*^k(A)$ the derived version, defined as the homology of the totalisation of the bicomplex $\CC_{\bullet}^k(P_\bullet)$, where $P_\bullet\to A$ is as in the previous paragraph. The usual SBI sequence remains valid in the derived setting: \[\cdots\To \HH_n^k(A)\stackrel{I}{\To} \HC_n^k(A)\stackrel{S}{\To} \HC_{n-2}^k(A)\stackrel{B}{\To}\cdots\]

The following standard lemma implies in particular that derived Hochschild and cyclic homology coincide with the usual theories when the base ring $k$ is a field:

\begin{lemma}\label{lemma_derived_HH_and_HC}
Let $k$ be a commutative ring, $A$ a $k$-algebra, and $M$ an $A$-bimodule. If $A$ is flat over $k$, then the canonical maps \[H_n^k(A,M)\To H_n^{\sub{naive},k}(A,M)\mbox{\quad and \quad}  \HC_n^k(A)\To \HC_n^{\sub{naive},k}(A)\] are isomorphisms for all $n\ge0$.
\end{lemma}
\begin{proof}
If $A$ is flat over $k$ then the map $M\otimes_kP_\bullet^{\otimes_kq}\to M\otimes_kA^{\otimes_kq}$ is a weak equivalence for all $q\ge 0$, and so the diagonal of the bisimplicial $k$-module $C_\bullet^k(P_\bullet,M)$ is weakly equivalent to $C_\bullet^k(A,M)$.
\end{proof}

More generally, the argument of Lemma \ref{lemma_derived_HH_and_HC} shows that $H_*^k(A,M)$ and $\HC_*^k(A)$ may be calculated using any degree-wise-flat simplicial resolution $P_\blob\to A$. In particular, if $A$ is a commutative $k$-algebra, then $P_\blob$ may be taken to be a simplicial resolution of $A$ by flat commutative $k$-algebras.

Given a $k$-algebra $A$ and ideal $I\subseteq A$, we may choose free simplicial $k$-algebra resolutions $P_\blob\to A$ and $Q_\blob\to A/I$ so that there is a compatible map $P_\blob\to Q_\blob$. Then the relative (derived) Hochschild homology groups $\HH_*^k(A,I)$ are defined to be the homotopy groups of the homotopy fibre of $C^k_\blob(P_\blob)\to C^k_\blob(Q_\blob)$ (or, equivalently, of $C^k_\blob(P_\blob,A)\to C^k_\blob(Q_\blob,A/I)$); by definition they fit into a long exact sequence \[\cdots\To \HH_n^k(A,I)\To \HH_n^k(A)\To \HH_n^k(A/I)\To\cdots\] Analogous comments apply to (derived) cyclic homology.

Our primary tool in this section will be a restriction spectral sequence for Hochschild homology. In topological Hochschild homology, and hence essentially also in derived Hochschild homology, it is due to M.~Brun \cite[Thm.~6.2.10--Rmk.~6.2.12]{Brun2000}. We present here a purely algebraic proof:

\begin{proposition}\label{proposition_HH_SS}
Let $k$ be a commutative ring, $A\to B$ a homomorphism of $k$-algebras, and $M$ a $B$-bimodule. Then there is a first quadrant spectral sequence of $k$-modules \[E^2_{pq}=H_p^k(B,\Tor_q^A(B,M))\Longrightarrow H_{p+q}^k(A,M).\]
\end{proposition}
\begin{proof}
We must show that the canonical map of simplicial $k$-modules \[M\dotimes_{A\dotimes_kA^\sub{\tiny op}}A\To (B\dotimes_AM)\dotimes_{B\dotimes_kB^\sub{\tiny op}}B\] is a weak equivalence. Indeed, it follows from the description of Hochschild homology given in Remark \ref{remark_enveloping_alg} that the homotopy of the right side is described by the desired spectral sequence, and that the homotopy of the left side is the desired abutment.

To describe the desired weak equivalence more explicitly, let $P_\blob\to A$ be a simplicial resolution of $A$ by free $k$-algebras, let $Q_\blob\to B$ be a simplicial resolution of $B$ by a free simplicial $P_\blob$-algebra (hence $Q_\blob$ is also a free simplicial $k$-algebra), and let $M_\blob\to M$ be a resolution of $M$ by a free simplicial $Q_\blob\otimes_kQ_\blob^\sub{op}$-module (hence $M_\blob$ is also a free simplicial $P_\blob\otimes_kP_\blob^\sub{op}$-module). We must show that the canonical map \[M_\blob\otimes_{P_\blob\otimes_kP_\blob^\sub{\tiny op}}P_\blob\To (Q_\blob\otimes_{P_\blob} M_\blob)\otimes_{Q_\blob\otimes_kQ_\blob^\sub{\tiny op}}Q_\blob\] is a weak equivalence. But in fact more is true: in each degree $p\ge0$ the canonical map \[M_p\otimes_{P_p\otimes_kP_p^\sub{\tiny op}}P_p\To (Q_p\otimes_{P_p} M_p)\otimes_{Q_p\otimes_kQ_p^\sub{\tiny op}}Q_p,\quad m\otimes a\mapsto (1\otimes m)\otimes a\] is an isomorphism of $k$-modules. Indeed, the inverse is easily seen to be given by $(b\otimes m)\otimes b'\mapsto mb'b\otimes 1$.
\end{proof}

\subsection{Pro excision in $\HH$, $\HC$, $\THH$, and $\TC$}\label{subsection_pro_excision_for_HH}
We now prove that pro Tor-unital ideals satisfy pro excision in derived Hochschild and cyclic homology, and in their topological counterparts. We begin with two corollaries of Proposition \ref{proposition_HH_SS}:

\begin{corollary}\label{Corollary_HH_1}
Let $k$ be a commutative ring, $A$ a $k$-algebra, and $I$ a pro Tor-unital ideal of $A$. Then the canonical map \[\{H_n^k(A,A/I^r)\}_r\To\{\HH_n^k(A/I^r)\}_r\] is an isomorphism of pro $k$-modules for all $n\ge0$.
\end{corollary}
\begin{proof}
Applying Proposition \ref{proposition_HH_SS} to the homomorphism $A\to A/I^r$ and bimodule $A/I^r$, and letting $r\to\infty$, we obtain a spectral sequence of pro $k$-modules \[E^2_{pq}(\infty)=\{H_p^k(A/I^r,\Tor_q^A(A/I^r,A/I^r))\}_r\implies \{H_{p+q}^k(A,A/I^r)\}_r.\] Since $I$ is pro Tor-unital, Theorem \ref{theorem1a} implies that $\{\Tor_q^A(A/I^r,A/I^r)\}_r=0$ for $q>0$, so the spectral sequence collapses to the desired edge map isomorphisms $\{H_n^k(A,A/I^r)\}_r\isoto\{\HH_n^k(A/I^r)\}_r$.
\end{proof}

We remark that $H_n^k(A,I)$, which appears in the following corollary and theorem, denotes the Hochschild homology of $A$ with coefficients in the $A$-bimodule $I$; this should not be confused with the relative Hochschild homology $\HH_n^k(A,I)$.

\begin{corollary}\label{Corollary_HH_2}
Let $k$ be a commutative ring, $A\to B$ a homomorphism of $k$-algebras, and $I$ a pro Tor-unital ideal of $A$ mapped isomorphically to an ideal of $B$. Then the canonical map \[\{H_n^k(A,I^r)\}_r\To \{H_n^k(B,I^r)\}_r\] is an isomorphism of pro $k$-modules for all $n\ge0$.
\end{corollary}
\begin{proof}
Applying Proposition \ref{proposition_HH_SS} to the homomorphism $A\to B$ and bimodule $I^r$, and letting $r\to\infty$, we obtain a spectral sequence of pro $k$-modules \[E^2_{pq}(\infty)=\{H_p^k(B,\Tor_q^A(B,I^r))\}_r\implies\{H_{p+q}^k(A,I^r)\}_r\] This will collapse to the desired edge map isomorphisms if we can prove that \[\{\Tor_q^A(B,I^r)\}_r=\begin{cases} \{I^r\}_r & q=0\\ 0 & q>0.\end{cases}\tag{\dag}\] Using the long exact $\Tor^A(B,-)$ sequence for $0\to I^r\to A\to A/I^r\to 0$, it is easy to see that (\dag) is equivalent to the condition that $\{\Tor_q^A(B,A/I^r)\}_r=0$ for all $q>0$. But using the long exact $\Tor^A(-,A/I^r)$ sequence for the exact sequence $0\to A\to B\oplus A/I^r\to B/I^r\to 0$ (similarly to the proof of Proposition \ref{proposition_pro_H}), it is therefore enough to show that $\{\Tor_q^A(B/I^r,A/I^r)\}_r=0$ for all $q>0$; this is true by Corollary \ref{corollary_most_useful_pro_vanishing}.
\end{proof}

We can now prove pro excision for pro Tor-unital ideals in Hochschild and cyclic homology:

\begin{theorem}[Pro excision for derived $\HH$ and $\HC$]\label{theorem_pro_excision_for_HH_HC}
Let $k$ be a commutative ring, $A\to B$ a homomorphism of $k$-algebras, and $I$ a pro Tor-unital ideal of $A$ mapped isomorphically to an ideal of $B$. Then the canonical maps \[\{\HH_n^k(A,I^r)\}_r\To \{\HH_n^k(B,I^r)\}_r\quad\mbox{and}\quad\{\HC_n^k(A,I^r)\}_r\To \{\HC_n^k(B,I^r)\}_r\] are isomorphisms of pro $k$-modules for all $n\ge0$.
\end{theorem}
\begin{proof}
In the definition of relative Hochschild homology immediately before Proposition \ref{proposition_HH_SS}, it is clear that the map $C_\blob^k(P_\blob,A)\to C_\blob^k(Q_\blob,A/I)$ factors through the canonical map $C_\blob^k(P_\blob,A)\to C_\blob^k(P_\blob,A/I)$, whose homotopy fibre is $C_\blob^k(P_\blob,I)$. We therefore obtain a commutative diagram of $k$-modules
\[\xymatrix{
\cdots\ar[r] & H_n^k(A,I) \ar[r]\ar[d]^{(2)} & \HH_n^k(A) \ar[r]\ar@{=}[d] & H_n^k(A,A/I)\ar[d]^{(1)} \ar[r] & \cdots\\
\cdots\ar[r] & \HH_n^k(A,I) \ar[r] & \HH_n^k(A) \ar[r] & \HH_n^k(A/I) \ar[r] & \cdots
}\]
After replacing $I$ by $I^r$ and taking the limit over $r$, arrow (1) becomes an isomorphism of pro $k$-modules for all $n\ge0$ by Corollary \ref{Corollary_HH_1}. Hence arrow (2) also becomes an isomorphism of pro $k$-modules, i.e., \[\{H_n^k(A,I^r)\}_r\isoto\{\HH_n^k(A,I^r)\}_r.\]

This isomorphism is also true for $B$ in place of $A$, so to complete the proof for Hochschild homology we must show that $\{H_n^k(A,I^r)\}_r\cong \{H_n^k(B,I^r)\}_r$; but this is exactly Corollary \ref{Corollary_HH_2}.

The claim for $HC$ then follows in the usual way by repeatedly applying the five lemma to the limit over $r$ of the SBI sequences for the relative groups:
\[\xymatrix{
\cdots\ar[r] & \HH_n^k(A,I^r)\ar[r]\ar[d] & \HC_n^k(A,I^r)\ar[r]\ar[d] & \HC_{n-2}^k(A,I^r)\ar[r]\ar[d] & \cdots \\
\cdots\ar[r] & \HH_n^k(B,I^r)\ar[r] & \HC_n^k(B,I^r)\ar[r] & \HC_{n-2}^k(B,I^r)\ar[r] & \cdots
}\]
\end{proof}

The previous proof may be repeated for topological Hochschild and cyclic homology, as constructed in, e.g., \cite{Dundas2013}; since this is the only occurrence in the paper of these theories, we will assume that the interested reader is familiar with them and not provide further background. This result was proved originally by Geisser and Hesselholt, at least with $\bb Z/p\bb Z$ coefficients, via a lengthy argument in \cite[\S2]{GeisserHesselholt2006}:

\begin{theorem}[Pro excision for $\THH$ and $\TC$]\label{theorem_pro_THH}
Let $A\to B$ be a homomorphism of rings, and $I$ a pro Tor-unital ideal of $A$ mapped isomorphically to an ideal of $B$. Then the canonical maps \[\{\THH_n(A,I^r)\}_r\To \{\THH_n(B,I^r)\}_r\quad\mbox{and}\quad \{\TC_n^m(A,I^r;p)\}_r\To \{\TC_n^m(B,I^r;p)\}_r\] of pro abelian groups are isomorphisms for all $n\ge0$, $m\ge1$, and primes $p\ge2$.
\end{theorem}
\begin{proof}
As explained before the proof of Proposition \ref{proposition_HH_SS}, the analogous spectral sequence for topological Hochschild homology has already been established by Brun in \cite[Thm.~6.2.10]{Brun2000}: given a homomorphism of rings $A\to B$ and a $B$-bimodule $M$, there is a first quadrant spectral sequence of abelian groups \[E^2_{pq}=\THH_p(B,\Tor_q^A(B,M))\Longrightarrow \THH_{p+q}(A,M).\] The proofs of Corollary \ref{Corollary_HH_1}, Corollary \ref{Corollary_HH_2}, and  Theorem \ref{theorem_pro_excision_for_HH_HC} then carry over verbatim to $\THH$.

(Warning about bad notation: Unfortunately, in the $\THH$ literature the notation $\THH(A,I)$ is ambiguous; it could denote either $\THH$ for $A$ with coefficients in the bimodule $M=I$, as in the spectral sequence, or relative $\THH$ for the map $A\to A/I$, as in the statement of the theorem. In this notation, a key step in the proof is the isomorphism $\{\THH_n(A,I^r)\}_r\isoto\{\THH_n(A,I^r)\}_r$, where the left side concerns coefficients in the bimodule $I$ and the right side is a relative group! The reader who wants to write a clear proof will need to invent new notation, such as $T\!H(A,I)$ for topological Hochschild homology of $A$ with coefficients in the bimodule $I$.)

The claim for $\TC^m$ then follows by the standard inductive argument of passing through $\TR^m$ and the homotopy orbit spectra; it is explained between Thms.~2.1 and 2.2 of \cite{GeisserHesselholt2006}.
\end{proof}

\begin{remark}[Excision for Tor-unital rings]\label{remark_non_pro}
Let $k$ be a commutative ring, $A\to B$ a homomorphism of $k$-algebras, and $I$ a Tor-unital ideal of $A$ mapping isomorphically to an ideal of $B$. Here by {\em Tor-unital} we mean that the following equivalent (by Remark \ref{remark_non_pro}) conditions are satisfied: $\Tor^A_n(A/I,A/I)=0$ for $n>0$; $\Tor^B_n(B/I,B/I)=0$ for $n>0$; $\Tor^{k\ltimes I}_n(k,k)=0$ for $n>0$; or $\Tor_n^{\bb Z\ltimes I}(\bb Z,\bb Z)=0$ for $n>0$. Then $I^r=I$ for all $r\ge1$, and so Theorems \ref{theorem_pro_excision_for_HH_HC} and \ref{theorem_pro_THH} show that $I$ satisfies excision in derived Hochschild and cyclic homology, and their topological counterparts.
\end{remark}

\begin{remark}[General pro algebras]\label{remark_pro_algebra_version}
Theorems \ref{theorem_pro_excision_for_HH_HC} and \ref{theorem_pro_THH}, and their proofs, remain true more generally for pro ideals and pro algebras. More precisely, let $k$ be a commutative ring, $A_\infty\to B_\infty$ a strict map of pro $k$-algebras, and $I_\infty$ an ideal of $A_\infty$ such that each homomorphism $A_r\to B_r$ carries $I_r$ isomorphically to an ideal of $B_r$; such a situation was studied in Proposition \ref{proposition_pro_H}. Assuming that $I_\infty$ is pro Tor-unital in the sense that $\{\Tor_n^{\bb Z\ltimes I_r}(\bb Z,\bb Z)\}_r=0$ for all $n>0$, it follows that $\{\HH^k_n(A_r,I_r)\}_r\isoto\{\HH^k_n(B_r,I_r)\}_r$, and similarly for $\HC^k$, $\THH$, and $\TC^m$.
\end{remark}

\subsection{Converse statements: excision implies Tor-unitality}\label{section_converse}
If $k$ is a commutative ring and $I$ is a non-unital $k$-algebra, then Wodzicki \cite[Thm.~3.1]{Wodzicki1989} proved that $I$ satisfies excision in usual Hochschild homology if {\em and only if} $I$ is H-unital, i.e., $B_\blob^k(I)\otimes_kV$ is acyclic for every $k$-module $V$ (where $B_\blob^k(I)$ is the bar complex from Example \ref{example_quasi_unital}). In Section \ref{subsection_pro_excision_for_HH} we showed that (pro) Tor-unitality is sufficient for (pro) excision in derived Hochschild and cyclic homology, and in this section we mimic Wodzicki's result by proving that it is moreover necessary.

We begin in the non-pro setting, establishing the converse to Remark \ref{remark_non_pro}:

\begin{theorem}\label{theorem_converse}
Let $k$ be a commutative ring and $I$ a non-unital $k$-algebra. Assume that $I$ satisfies excision in derived Hochschild homology over $k$, i.e., whenever $A\to B$ is a homomorphism of $k$-algebras which compatibly contain $I$ as an ideal then $\HH_n^k(A,I)\isoto \HH_n^k(B,I)$ for all $n\ge0$. Then $I$ is Tor-unital.
\end{theorem}
\begin{proof}
We assume initially that $I$ is flat over $k$ and recall some arguments of Wodzicki (see also Loday's explanation \cite[Thm.~1.4.10]{Loday1992}). Flatness of $I$ implies that $\HH_*^k(k\ltimes I)$ is given by the homology of the normalised Hochschild complex \cite[\S1.1.14]{Loday1992}:
\[\res C_\blob^k(k\ltimes I):=(k\ltimes I)\otimes_kI^{\otimes_k\blob}=\qquad 0\leftarrow k\ltimes I\stackrel{b}{\leftarrow} (k\ltimes I)\otimes_kI\stackrel{b}{\leftarrow} (k\ltimes I)\otimes_kI\otimes_kI\stackrel{b}{\leftarrow}\cdots\]
For any non-unital, flat $k$-algebra $I'$ containing $I$ as a two-sided ideal and such that $I'/I$ is flat as a $k$-module, we write $K_\blob(I'):=\ker\big(\res C_\blob^k(k\ltimes I')\to \res C_\blob^k(k\ltimes I'/I)\big)$, whose homology is the relative Hochschild homology $\HH_*^k(k\ltimes I',I)$.

In particular, if $J$ is another non-unital, flat $k$-algebra, then we may take $I':=I\oplus J$ and form the following commutative diagram of short exact sequences of complexes of $k$-modules:
\[\xymatrix@C=8mm{
0\ar[r] & K_\blob(I\oplus J)\ar[r] &\res C_\blob^k(k\ltimes(I\oplus J))\ar[r] & \res C_\blob^k(k\ltimes J)\ar[r] & 0\\
0\ar[r] &K_\blob(I)\ar[r]\ar[u] &\res C_\blob^k(k\ltimes I)\ar[r]\ar[u] & \res C_\blob^k(k)\ar[r]\ar[u] & 0
}\]
The vertical arrows are naturally split. Moreover, it is not hard to see that $K_\blob(I\oplus J)$ admits a natural direct sum decomposition into a number of subcomplexes, which include a copy of $K_\blob(I)$ and a copy of a shift of $B_\blob^k(I)\otimes_kJ$, where $B_\blob^k(I)$ is the bar complex from Example \ref{example_pro_H}. In particular, fixing a free, rank-one $k$-module $V$ with zero multiplication map, there is a natural decomposition \[K_\blob(I\oplus V)=K_\blob(I)\oplus B_{\blob-1}^k(I)\oplus K_\blob'(I),\] where $K_\blob'(I)$ is a certain complex of $k$-modules whose precise structure is unimportant.

We now no longer suppose that $I$ is flat over $k$. Let $P_\blob\to k\ltimes I$ be a simplicial resolution of $k\ltimes I$ by free $k$-algebras; the composition $P_\blob\to k\ltimes I\to k$ shows that $k\to P_\blob$ is canonically split, and so we may write $P_\blob =k\ltimes I_\blob$ in such a way that $I_\blob\to I$ is a simplicial resolution of $I$ by non-unital, flat $k$-algebras. Moreover, $k\ltimes (I_\blob\oplus V)\to k\ltimes(I\oplus V)$ is a simplicial resolution of $k\ltimes(I\oplus V)$ by flat $k$-algebras.

According to the argument above in the flat case, there is a natural decomposition of bicomplexes of $k$-modules \[K_\blob(I_\blob\oplus V)=K_\blob(I_\blob)\oplus B_{\blob-1}^k(I_\blob)\oplus K'_\blob(I_\blob).\] Note that $\op{Tot}K_\blob(I_\blob\oplus V)$ calculates $\HH_*^k(k\ltimes (I\oplus V),I)$ and that $\op{Tot}K_\blob(I_\blob)$ calculates $\HH_*^k(k\ltimes I,I)$. Our assumption that $I$ satisfies excision implies that the canonical inclusion $\HH_n^k(k\ltimes I,I)\to \HH_n^k(k\ltimes (I\oplus V),I)$ is an isomorphism for all $n\ge0$, and so we conclude that $\op{Tot}B_\blob^k(I_\blob)$ (and $\op{Tot}K'_\blob(I_\blob)$) must be acyclic. But $\op{Tot}B_\blob^k(I_\blob)$ calculates $\Tor^{k\ltimes I}_*(k,I)=\Tor^{k\ltimes I}_{*+1}(k,k)$, which completes the proof.
\end{proof}

The previous proof extends without difficulty to the case of pro algebras, thereby establishing the converse of our general pro-excision statement given in Remark \ref{remark_pro_algebra_version}:

\begin{theorem}
Let $k$ be a commutative ring and $I_\infty$ a pro non-unital $k$-algebra. Assume that $I_\infty$ satisfies excision in derived Hochschild homology over $k$, i.e., whenever $A_\infty\to B_\infty$ is a strict map of pro $k$-algebras which compatibly contain $I_\infty$ as an ideal (as in Rmk.~\ref{remark_pro_algebra_version}) then $\{\HH^k_n(A_r,I_r)\}_r\isoto\{\HH^k_n(B_r,I_r)\}_r$ for all $n\ge0$. Then $I_\infty$ is pro Tor-unital.
\end{theorem}
\begin{proof}
We  use the notation and arguments of the proof of Theorem \ref{theorem_converse}. We will prove the theorem by considering the morphism of pro $k$-algebras \[A_\infty=\{k\ltimes I_r\}_r\To \{k\ltimes (I_r\oplus V)\}_r=B_\infty.\]

Let $P^{(r)}_\blob\to k\ltimes I_r$ be a simplicial resolution of $k\ltimes I_r$ by free $k$-algebras, chosen sufficiently functorially so that there is a resulting morphism of pro simplicial $k$-algebras $\{P^{(r)}_\blob\}_r\to \{k\ltimes I_r\}_r$. We may write $P^{(r)}_\blob=k\ltimes I^{(r)}_\blob$, where $I_\blob^{(r)}\to I_r$ is a simplicial resolution of $I_r$ by non-unital, flat $k$-algebras. Then the exact same argument as Theorem \ref{theorem_converse} shows that $\{H_n(\op{Tot}B_\blob^k(I_\blob^{(r)}))\}_r$ vanishes for all $n\ge0$; but $H_n(B_\blob^k(I_\blob^{(r)}))=\Tor^{k\ltimes I_r}_{n+1}(k,k)$ for all $r\ge1$, so this completes the proof.
\end{proof}

\begin{remark}
Our main interest in this paper is not arbitrary pro algebras, but rather powers of a constant ideal in a constant algebra. From this point of view it is natural to ask the following question:
\begin{quotation}
Let $k$ be a commutative ring and $I$ a non-unital $k$-algebra. Assume that $\{I^r\}_r$ satisfies excision in derived Hochschild homology over $k$ {\em only for constant algebras}, i.e., whenever $A\to B$ is a homomorphism of $k$-algebras which compatibly contain $I$ as an ideal, then $\{\HH^k(A,I^r)\}_r\isoto\{\HH^k(B,I^r)\}_r$ for all $n\ge0$. Is $I$ pro Tor-unital?
\end{quotation}
We do not know the answer to this question.
\end{remark}

\subsection{Cuntz--Quillen theorem over general base rings}\label{subsection_HP}
The celebrated Cuntz--Quillen theorem \cite{CuntzQuillen1997} states that periodic cyclic homology over a field $k$ of characteristic zero satisfies excision. In this section we extend their result to arbitrary commutative base rings containing $\bb Q$.

We begin by discussing the derived version of periodic cyclic homology. Given a commutative ring $k$ and a $k$-algebra $A$, let $\HP^{\sub{naive},k}_*(A)$ denote the usual periodic cyclic homology of $A$ as a $k$-algebra, defined as the homology of $\holim_s\CC^k_{\blob+2s}(A)$, where the homotopy limit is taken over repeated applications of the periodicity map $S:\CC^k_{\blob+2}(A)\to \CC^k_{\blob}(A)$. As for Hochschild and cyclic homology, we denote by $\HP^k_*(A)$ the derived version, defined by replacing $\CC^k_{\blob}(A)$ by $\CC^k_{\blob}(P_\blob)$, where $P_\blob\to A$ is a simplicial resolution of $A$ by free $k$-algebras; to be precise, since the order of totalisations and homotopy limits may cause confusion, $\HP_*^k(A)$ is defined to be the homology of the unbounded complex ${\holim}_s\op{Tot}\CC_{\blob+2s}(P_\blob)$. The canonical maps $\HP^k_*(A)\to\HP^{\sub{naive},k}_*(A)$ are isomorphisms if $A$ is flat over $k$.
 
Although we do not need it to prove Theorem \ref{theorem_Cuntz-Quillen}, we mention now that derived periodic cyclic homology continues to satisfy Goodwillie's nil-invariance property \cite[Thm.~II.5.1]{Goodwillie1985}:

\begin{lemma}\label{lemma_derived_Goodwillie}
Let $k$ be a commutative $\bb Q$-algebra, $A$ a $k$-algebra, and $I$ a nilpotent ideal of $A$. Then the canonical map $\HP_n^k(A)\to\HP_n^k(A/I)$ is an isomorphism for all $n\in\bb Z$.
\end{lemma}
\begin{proof}
This may be proved by reduction to the flat case, where it follows from Goodwillie's original result; see \cite[Corol.~6.5.2.5]{Dundas2013} for more details.
\end{proof}

The following is the main theorem of this section:

\begin{theorem}[Excision in derived periodic cyclic homology]\label{theorem_Cuntz-Quillen}
Let $k$ be a commutative $\bb Q$-algebra. Then derived period cyclic homology over $k$ satisfies excision; i.e., if $A\to B$ is a homomorphism of $k$-algebras, and $I$ is an ideal of $A$ mapped isomorphically to an ideal of $B$, then the canonical map \[\HP_n^k(A,I)\To\HP_n^k(B,I)\] is an isomorphism for all $n\in\bb Z$.
\end{theorem}
\begin{proof}
We first sketch the main ideas of the proof. Theorem \ref{theorem_pro_excision_for_HH_HC} and Lemma \ref{lemma_derived_Goodwillie} easily imply that derived periodic cyclic homology over $k$ satisfies excision whenever $I$ is a pro Tor-unital ideal. The key additional step is therefore to show that pro Tor-unital ideals are sufficiently abundant to deduce excision in general; this is the content of Proposition \ref{proposition_pro_Tor_unital_in_free}(ii) below, in which it is shown that certain ideals of free $k$-algebras are automatically pro Tor-unital. Unfortunately, the ``standard argument'' to reduce excision to the free case does not work in our situation (see Remark \ref{remark_standard_reduction_to_free} for a more precise statement), thereby complicating the technical details of the proof.

We begin by introducing some relative and birelative notation in cyclic homology; for clarity we will omit the base ring $k$ from all notation. Given an ideal $J$ of a $k$-algebra $R$, let $\CC_\blob(R,J):=\ker(\CC_\blob(R)\to\CC_\blob(R/J))$ denote the complex defining relative cyclic homology $\HC^\sub{naive}_*(R,J)$. If $R\to S$ is a homomorphism of $k$-algebras mapping $J$ isomorphically to an ideal of $S$, then $\CC_\blob(R,S,J):=\op{hofib}(\CC_\blob(R,J)\to\CC_\blob(S,J))$ denotes the complex defining the birelative cyclic homology groups $\HC_*^\sub{naive}(R,S,J)$. Thus the two columns and row of the diagram
\[\xymatrix{
\CC_\blob(R,S,J) \ar[r] & \CC_\blob(R,J)\ar[r]\ar[d] & \CC_\blob(S,J)\ar[d]\\
&\CC_\blob(R)\ar[r]\ar[d] &\CC_\blob(S)\ar[d]\\
&\CC_\blob(R/J)\ar[r]&\CC_\blob(S/J)
}\]
are fibre sequences. Note also that all these complexes may be successively shifted leftwards by $2s$, for $s\ge1$, and the limit over the periodicity maps taken to obtain a similar diagram concerning periodic cyclic homology.

We now properly begin the proof; we suppose that $I$ is an ideal of a $k$-algebra $A$ and will show that $\HP_n(k\ltimes I,I)\isoto\HP_n(A,I)$ for all $n\in\bb Z$, which is sufficient. Let $P_\blob\to A$ and $Q_\blob\to A/I$ be simplicial resolutions by free $k$-algebras which are compatible in that they fit into a commutative diagram
\[\xymatrix{
P_\blob\ar[r]\ar[d] & Q_\blob\ar[d] \\
A\ar[r] & A/I
}\]
where $P_\blob\to Q_\blob$ is degree-wise surjective (for example, first pick any free resolution $Q_\blob\to A/I$ and then let $P_\blob$ be a free resolution of the pull-back $A\times_{A/I}Q_\blob$); write $I_\blob:=\ker(P_\blob\to Q_\blob)$. For each fixed $p\ge0$, our remarks on relative and birelative cyclic homology may be applied to the homomorphism $k\ltimes I_p\to P_p$ and ideal $I_p$ to yield the diagram
\[\xymatrix{
\CC_\blob(k\ltimes I_p,P_p,I_p) \ar[r] & \CC_\blob(k\ltimes I_p,I_p)\ar[r]\ar[d] & \CC_\blob(P_p,I_p)\ar[d]\\
&\CC_\blob(k\ltimes I_p)\ar[r]\ar[d] &\CC_\blob(P_p)\ar[d]\\
&\CC_\blob(k)\ar[r]&\CC_\blob(Q_p)
}\]
We make the following claim concerning the homology of the top left of this diagram:
\begin{quote}
For any $p,M\ge0$, there exists $N>0$ such that the iterated periodicity map $S^N:\HC_{n+2N}^\sub{naive}(k\ltimes I_p,P_p,I_p)\to\HC_n^\sub{naive}(k\ltimes I_p,P_p,I_p)$ is zero for $0\le n\le M$.
\end{quote}

Before proving the claim we explain why it is sufficient to complete the proof. Let $\CC_\blob(k\ltimes I_\blob,P_\blob,I_\blob)$ denote the simplicial complex $p,q\mapsto \CC_q(k\ltimes I_p,P_p,I_p)$; since only finitely many values of $p$ affect the homology of $\op{Tot}\CC_\blob(k\ltimes I_\blob,P_\blob,I_\blob)$ in any fixed degree, the claim shows that for each $n\ge0$ there exists $N>0$ such that \[S^N:\op{Tot}\CC_{\blob+2N}(k\ltimes I_\blob,P_\blob,I_\blob)\to \op{Tot}\CC_\blob(k\ltimes I_\blob,P_\blob,I_\blob)\] induces zero on degree $n$ homology; hence $\op{holim}_s\op{Tot}\CC_{\blob+2s}(k\ltimes I_\blob,P_\blob,I_\blob)\simeq 0$. Totalising the complexes of the previous diagram and taking the limit over the periodicity maps, it follows that the square
\[\xymatrix{
\op{holim}_s\op{Tot}\CC_{\blob+2s}(k\ltimes I_\blob)\ar[r]\ar[d] & \op{holim}_s\op{Tot}\CC_{\blob+2s}(P_\blob)\ar[d]\\
\op{holim}_s\op{Tot}\CC_{\blob+2s}(k)\ar[r] & \op{holim}_s\op{Tot}\CC_{\blob+2s}(Q_\blob)
}\]
is homotopy cartesian. But $P_\blob\to A$ and $Q_\blob\to A/I$ are resolutions by free $k$-algebras, and $k\ltimes I_\blob\to k\ltimes I$ is a resolution by flat $k$-algebras (by construction $I_\blob\simeq I$, and each $I_p$ is a flat $k$-module since the same is true of $P_p$ and $Q_p=P_p/I_p$); the fact that this diagram is homotopy cartesian is exactly the desired excision statement, completing the proof.

It remains to prove the claim; so fix $p,M\ge0$. It is convenient to consider the following diagram for any $r\ge1$, in which all rows and columns are fibre sequences:
\[\xymatrix{
\CC_\blob(k\ltimes I_p,P_p,I_p^r)\ar[r]\ar[d]_{(1)} & \CC_\blob(k\ltimes I_p,I_p^r)\ar[r]\ar[d]&\CC_\blob(P_p,I_p^r)\ar[d]\\
\CC_\blob(k\ltimes I_p,P_p,I_p)\ar[r]\ar[d]_{(2)} & \CC_\blob(k\ltimes I_p,I_p)\ar[r]\ar[d]&\CC_\blob(P_p,I_p)\ar[d]\\
\CC_\blob(k\ltimes I_p/I_p^r,P_p/I_p^r,I_p/I_p^r)\ar[r] & \CC_\blob(k\ltimes I_p/I_p^r,I_p/I_p^r)\ar[r]&\CC_\blob(P_p/I_p^r,I_p/I_p^r)
}\]
Since $I_p$ is an ideal of the free $k$-algebra $P_p$ such that $P_p/I_p=Q_p$ is also a free $k$-algebra (hence certainly projective as a $k$-module), Proposition \ref{proposition_pro_Tor_unital_in_free}(ii) below implies that $I_p$ is pro Tor-unital. Hence Theorem \ref{theorem_pro_excision_for_HH_HC} implies that $\{\HC_n^\sub{naive}(k\ltimes I_p,P_p,I_p^r)\}_r=0$ for all $n\ge0$ (note that the $k$-algebras $k\ltimes I_p$, $k\ltimes I_p/I_p^r$, $P_p$, $P_p/I_p$, $P_p/I_p^r$ are all flat, using Proposition \ref{proposition_pro_Tor_unital_in_free}(i) below, so the relevant derived and naive cyclic homologies coincide). That is, there exists $r>0$ such that map (1) induces zero on degree $n$ homology for $0\le n\le M$, and so map (2) induces an injection on homology in the same range.

Next, according to Goodwillie's proof of the nil-invariance of periodic cyclic homology \cite[Thm.~II.5.1]{Goodwillie1985}, the iterated periodicity maps 
\begin{align*}
S^{(r-1)(n+1)+1}:&\HC_{n+2(r-1)(n+1)+2}^\sub{naive}(k\ltimes I_p/I_p^r,I_p/I_p^r)\to\HC_n^\sub{naive}(k\ltimes I_p/I_p^r,I_p/I_p^r)\\
S^{(r-1)(n+1)+1}:& \HC_{n+2(r-1)(n+1)+2}^\sub{naive}(P_p/I_p^r,I_p/I_p^r)\to\HC_n^\sub{naive}(P_p/I_p^r,I_p/I_p^r)
\end{align*}
are zero for all $n\ge0$. From the bottom row of the previous diagram it therefore follows that there exists $N>0$ such that \[S^N:\HC_{n+2N}^\sub{naive}(k\ltimes I_p/I_p^r,P_p/I_p^r,I_p/I_p^r)\To \HC_n^\sub{naive}(k\ltimes I_p/I_p^r,P_p/I_p^r,I_p/I_p^r)\] is zero for $0\le n\le M$. By injectivity of (2) on homology, we then deduce that \[S^N:\HC_{n+2N}^\sub{naive}(k\ltimes I_p,P_p,I_p)\To \HC_n^\sub{naive}(k\ltimes I_p,P_p,I_p)\] is zero for $0\le n\le M$, proving the claim.
\end{proof}

\begin{corollary}
Let $k$ be a commutative $\bb Q$-algebra, $A\to B$ a homomorphism of $k$-algebras, and $I$ an ideal of $A$ mapped isomorphically to an ideal of $B$; assume that $A$, $B$, $A/I$, and $B/I$ are flat $k$-modules. Then the canonical map $\HP_n^{\sub{naive},k}(A,I)\to\HP_n^{\sub{naive},k}(B,I)$ is an isomorphism for all $n\in\bb Z$.
\end{corollary}
\begin{proof}
This is a special case of Theorem \ref{theorem_Cuntz-Quillen}, since the flatness assumptions imply that the relevant derived and naive periodic cyclic homologies coincide.
\end{proof}

The following result, implying that ``enough'' pro Tor-unital ideals exist over an arbitrary commutative ring, supplies the missing step to the proof of Theorem \ref{theorem_Cuntz-Quillen}:

\begin{proposition}\label{proposition_pro_Tor_unital_in_free}
Let $k$ be a commutative ring, $P$ a free $k$-algebra, and $I$ an ideal of $P$ such that $P/I$ is projective as a $k$-module. Then:
\begin{enumerate}
\item $P/I^r$ is a projective $k$-module for all $r\ge1$.
\item $I$ is pro Tor-unital.
\end{enumerate}
\end{proposition}
\begin{proof}
Since $P$ is a free $k$-algebra, it is known that $\Omega^1P:=\ker(P\otimes_kP\xto{\sub{mult.}}~P)$ is projective as a $P$-bimodule; see, e.g., the proof of \cite[Prop.~9.1.6]{Weibel1994}. Hence the argument of \cite[Prop.~5.1]{CuntzQuillen1995} proves the following: if $M$ is a left $P$-module which is projective as a $k$-module, then $M$ has projective dimension $\le1$ over $P$. Our assumption therefore implies that $P/I$ has projective dimension $\le 1$ over $P$, whence $I$ is projective as a left $P$-module and $0\to I\to P\to P/I\to 0$ is a resolution of $P/I$ by projective left $P$-modules. So \[\Tor_n^P(P/I,P/I)=\begin{cases} I/I^2&n=1\\0&n>1\end{cases}\]

(i): The case $r=1$ is assumed, so we may proceed by induction and assume that $P/I^r$ is a projective $k$-module. From the previous paragraph it follows that $P/I^r\otimes_PI=I/I^{r+1}$ is projective as a left $P/I^r$-module, hence also projective as a $k$-module by the inductive hypothesis. So the outer terms in the short exact sequence $0\to I/I^{r+1}\to P/I^{r+1}\to P/I\to 0$ are projective $k$-modules, whence the central term is also.

(ii): By (i), the first paragraph of the proof applies to $I^r$ in place of $I$, proving \[\Tor_n^P(P/I^r,P/I^r)=\begin{cases} I^r/I^{2r}&n=1\\0&n>1\end{cases}\] Therefore the map $\Tor_n^P(P/I^{2r},P/I^{2r})\to \Tor_n^P(P/I^r,P/I^r)$ is zero for all $n,r\ge1$, completing the proof.
\end{proof}

\begin{remark}\label{remark_standard_reduction_to_free}
There is a certain argument which is now well-known for reducing excision to the case of ideals of free algebras; e.g., the proof of \cite[Thm~5.3]{CuntzQuillen1997}, \cite[Lem.~1.4]{Cortinas2006}, or \cite[Lem.~3.1]{GeisserHesselholt2006}. If one combines this argument with Proposition \ref{proposition_pro_Tor_unital_in_free}, Lemma \ref{lemma_derived_Goodwillie}, and Theorem \ref{theorem_pro_excision_for_HH_HC}, then it proves Theorem \ref{theorem_Cuntz-Quillen} only in the special case that $A$, $B$, $A/I$, and $B/I$ are all projective as $k$-modules. 
\end{remark}

\begin{remark}
It may be worth observing that Theorem \ref{theorem_Cuntz-Quillen} reproves Cuntz--Quillen's excision theorem over fields without using their quasi-unitality results discussed in Remark \ref{example_quasi_unital}.
\end{remark}

\begin{remark}[Quasi-free algebras]
If $k$ is a commutative ring, and $P$ is a $k$-algebra such that the structure map $k\to P$ is injective and the quotient $P/k$ is projective as a $k$-module, then the arguments of \cite[\S1--3]{CuntzQuillen1995} show that the following two conditions are equivalent:
\begin{enumerate}
\item Any $k$-algebra surjection $B\onto P$ with square-zero kernel has a $k$-algebra splitting.
\item $\Omega^1P:=\ker(P\otimes_kP\xto{mult.}P)$ is projective as a $P$-bimodule.
\end{enumerate}
$P$ is said to be a {\em quasi-free} $k$-algebra if and only if these conditions are satisfied; e.g., free $k$-algebras are quasi-free. It is clear from the proof that Proposition \ref{proposition_pro_Tor_unital_in_free} remains true for quasi-free $k$-algebras.
\end{remark}

\subsection{Continuity properties for $\HH$ of commutative rings}\label{section_HH_and_HC}
When $A$ is a commutative $k$-algebra, the Hochschild homology groups $\HH_n^k(A)$ are $A$-modules; this allows more continuity results in the style of Corollaries \ref{Corollary_HH_1} and \ref{Corollary_HH_2} to be established. After the provisional Lemma \ref{lemma_HH_commutative}, we give three applications: Artin--Rees properties, the pro HKR theorem, and the Fe\u\i gin--Tsygan theorem.

We stress that, in the next two results, the pro Tor-unitality assumption on the ideal $I$ is satisfied as soon as $A$ is Noetherian, by Theorem \ref{theorem_pro_Tor_unital_commutative}.

\begin{lemma}\label{lemma_HH_commutative}
Let $k\to A$ be a homomorphism of commutative rings, $I$ a pro Tor-unital ideal of $A$, and $M$ an $A$-module. Then the canonical map \[\{H_n^k(A,M/I^rM)\}_r\To \{H_n^k(A/I^r,M/I^rM)\}_r\] is an isomorphism of pro $A$-modules for all $n\ge0$.
\end{lemma}
\begin{proof}
Applying Proposition \ref{proposition_HH_SS} to the homomorphism $A\to A/I^r$ and module $M/I^rM$, and letting $r\to\infty$, we obtain a spectral sequence of pro $A$-modules \[E^2_{pq}(\infty)=\{H_p^k(A/I^r,\Tor_q^A(A/I^r,M/I^rM))\}_r\implies \{H_{p+q}^k(A,M/I^rM)\}_r.\] Since $I$ is pro Tor-unital, Corollary \ref{corollary_most_useful_pro_vanishing} implies that $\{\Tor_q^A(A/I^r,M/I^rM)\}_r=0$ for $q>0$, so the spectral sequence collapses to the desired edge map isomorphisms $\{H_n^k(A,M/I^rM)\}_r\isoto\{H_n^k(A/I^r,M/I^rM)\}_r$.
\end{proof}

{\bf Artin--Rees properties.} From Lemma \ref{lemma_HH_commutative} we obtain fundamental Artin--Rees vanishing results for derived Hochschild and cyclic homology, generalising results for Andr\'e--Quillen homology in the Noetherian case \cite[Prop.~X.12]{Andre1974} \cite[Thm.~6.15]{Quillen1970}; further generalisations to Andr\'e--Quillen homology will be given in Section \ref{section_AR_for_AQ}:

\begin{theorem}\label{theorem_AR_for_HH}
Let $A$ be a commutative ring, and $I$ a pro Tor-unital ideal of $A$. Then:
\begin{enumerate}
\item $\{\HH_n^A(A/I^r)\}=0$ for all $n>0$.
\item $\{\HC_n^A(A/I^r)\}=0$ for all odd $n>0$, while in even degrees the periodicity maps \[\cdots\stackrel{S}{\To}\{\HC_{2n}^A(A/I^r)\}_r\stackrel{S}{\To}\cdots\stackrel{S}{\To}\{\HC_2^A(A/I^r)\}_r\stackrel{S}{\To}\{\HC_0^A(A/I^r)\}_r\cong\{A/I^r\}_r\] are all isomorphisms.
\end{enumerate}
\end{theorem}
\begin{proof}
(i): By Corollary \ref{Corollary_HH_1}, with $k=A$, it is sufficient to show that $\{H_n^A(A,A/I^r)\}_r=0$ for $n>0$. But for any $A$-module $M$, it is clear that $H_n^A(A,M)=0$ for $n>0$.

(ii): This follows from part (i) and the SBI sequence.
\end{proof}

{\bf Pro HKR theorem.} Recall that the Hochschild--Kostant--Rosenberg theorem \cite[Thm.~3.4.4]{Loday1992} states that if $k\to A$ is a smooth morphism of commutative, Noetherian rings, then the canonical antisymmetrisation map $\Omega_{A|k}^n\to\HH_n^k(A)$ is an isomorphism for all $n\ge0$; by N\'eron--Popescu desingularisation \cite{Popescu1985, Popescu1986} this holds more generally if $k\to A$ is merely assumed to be geometrically regular. (Here we follow Swan's notation of saying that a morphism $k\to A$ is {\em geometrically regular} if and only if it is flat and has geometrically regular fibres in the usual sense \cite{Swan1998}; other authors prefer to say more briefly that the morphism is ``regular''.) The next result is a pro version of the HKR theorem. A proof for finite type algebras over fields can be found in \cite[Thm.~3.2]{Cortinas2009}, but for recent applications to the formal deformation of algebraic cycles \cite{BlochEsnaultKerz2013, Morrow_Deformational_Hodge} the following stronger version is required:

\begin{theorem}\label{theorem_pro_HKR}
Let $k\to A$ be a geometrically regular morphism of commutative, Noetherian rings, and let $I$ be an ideal of $A$. Then the antisymmetrisation map \[\{\Omega_{A/I^r|k}^n\}_r\To\{\HH_n^k(A/I^r)\}_r\] is an isomorphism of pro $A$-modules for all $n\ge0$.
\end{theorem}
\begin{proof}
Recall that $I$ is pro Tor-unital by Theorem \ref{theorem_pro_Tor_unital_commutative}. So, by Lemma \ref{lemma_HH_commutative}, the canonical map $\{H_n^k(A,A/I^r)\}_r\to \{\HH_n^k(A/I^r)\}_r$ is an isomorphism. But since $k\to A$ is geometrically regular, the classical HKR theorem implies that $H_n^k(A,A/I^r)\cong {\Omega_{A|k}^n\otimes_AA/I^r}$. Finally, the isomorphism $\{\Omega_{A|k}^n\otimes_AA/I^r\}_r\cong\{\Omega_{A/I^r|k}^n\}$ is an easy consequence of the inclusion $d(I^{2r})\subseteq I^r\Omega_{A|k}^1$.
\end{proof}

\begin{remark}
More generally, if $k\to A$ and $I\subseteq A$ satisfy the conditions of Theorem \ref{theorem_pro_HKR}, and if $M$ is an $A$-module, then the same argument shows that the antisymmetrisation map $\{\Omega_{A/I^r|k}^n\otimes_{A/I^r}M/I^rM\}_r\to\{H_n^k(A/I^r,M/I^rM)\}_r$ is an isomorphism for all $n\ge0$.
\end{remark}

{\bf Fe\u\i gin--Tsygan theorem.} If $R$ is a commutative, finitely generated algebra over a characteristic zero field $k$, then the Fe\u\i gin--Tsygan theorem \cite{FeiginTsygan1985} identifies $\HP_n^k(R)$ with $\prod_{i\in\bb Z}H_\sub{crys}^{2i-n}(R)$, where $H_\sub{crys}^*(R)$ denotes Grothendieck's crystalline cohomology \cite{Grothendieck1968} (aka.~Hartshorne's algebraic de Rham cohomology \cite{Hartshorne1975}), defined as follows: let $R=A/I$ be a representation of $R$ as a quotient of a smooth $k$-algebra $A$ by an ideal $I$, let $\hat \Omega_{A|k}^\blob:=\projlim_r\Omega_{A/I^r|k}^\blob$ be the $I$-adic completion of the de Rham complex of $A$, and set $H^*_\sub{crys}(R):=H^*(\hat \Omega_{A|k}^\blob)$.

The following generalises the Fe\u\i gin--Tsygan theorem, as well as reproving the classical case:

\begin{theorem}\label{theorem_Feigin-Tsygan}
Let $k\to A$ be a geometrically regular morphism of commutative, Noetherian $\bb Q$-algebras, $I$ an ideal of $A$, and set $R:=A/I$. Then there is a natural isomorphism of $k$-modules \[\HP^k_n(R)\cong\prod_{i\in\bb Z}H^{2i-n}(\hat\Omega_{A|k}^\blob)\] for all $n\in\bb Z$, where $\hat\Omega^\blob_{A|k}:=\projlim_r\Omega^\blob_{A/I^r|k}$.
\end{theorem}
\begin{proof}
Arguing informally and ignoring $\projlim^1$ terms, we sketch the proof. The following leftwards map is an isomorphism by nil-invariance of periodic cyclic homology, and the rightwards map is an isomorphism by an $\HP$ version of Theorem \ref{theorem_pro_HKR}:
\[\HP_n^k(R)\longleftarrow \projlim_r\HP_n^k(A/I^r)\To \prod_{i\in\bb Z}H^{2i-n}(\hat\Omega_{A|k}^\blob)\]
We will now formalise this argument using the language of mixed complexes \cite[\S2.5.13]{Loday1992}.

For any $k$-algebra there is a natural map of mixed complexes $\pi:(C_\blob^k(-),b,B)\to(\Omega_{-|k}^\blob,0,d)$ which splits the anti-symmetrisation map $\Omega_{-|k}^*\to\HH_*^k(-)$ on the associated Hochschild homologies \cite[\S2.3]{Loday1992}. Applying this to each $k$-algebra $A/I^r$ and letting $r\to\infty$ yields a map of mixed complexes $(\projlim_rC_\blob^k(A/I^r),b,B)\to(\hat\Omega_{A|k}^\blob,0,d)$.

Let $P_\blob^{(r)}\to A/I^r$ be simplicial resolutions by free $k$-algebras, chosen compatibly for all $r\ge1$. Then the composition \[(\op{holim}_r\op{Tot}C_\blob^k(P_\blob^{(r)}),b,B)\To(\projlim_rC_\blob^k(A/I^r),b,B)\To(\hat\Omega_{A|k}^\blob,0,d)\] is an isomorphism on the associated Hochschild homologies by Theorem \ref{theorem_pro_HKR}, hence also induces an isomorphism on the associated periodic cyclic homologies \[\HP_*(\op{holim}_r\op{Tot}C_\blob^k(P_\blob^{(r)}),b,B)\isoto \HP_*(\hat\Omega_{A|k}^\blob,0,d).\]

Since $(\hat\Omega_{A|k}^\blob,0,d)$ has $B$-differential equal to zero, a standard argument shows that $\HP_n(\hat\Omega_{A|k}^\blob,0,d)=\prod_{i\in\bb Z}H^{2i-n}(\hat\Omega_{A|k}^\blob)$ for all $n\in\bb Z$. The proof will therefore be complete as soon as we show that the canonical map \[\HP_n(\op{holim}_r\op{Tot}C_\blob^k(P_\blob^{(r)}),b,B)\To \HP_n(\op{Tot}C_\blob^k(P_\blob^{(1)}),b,B)=\HP_n^k(R)\] is an isomorphism for all $n\in\bb Z$; i.e., that \[\op{holim}_s\op{holim}_r\op{Tot}\CC_{\blob+2s}^k(P_\blob^{(r)})\To\op{holim_s}\op{Tot}\CC_{\blob+2s}^k(P_\blob^{(1)})\] is a weak equivalence, in which each $\op{holim}_s$ is taken over the periodicity maps. But by interchanging the order of the homotopy limits on the left, this is a consequence of Lemma \ref{lemma_derived_Goodwillie}, which states that $\op{holim}_s\op{Tot}\CC_{\blob+2s}^k(P_\blob^{(r)})\quis\op{holim_s}\op{Tot}\CC_{\blob+2s}^k(P_\blob^{(1)})$ for all $r\ge1$.
\end{proof}

\begin{remark}[Homotopy invariance of derived $\HP$]
If $k$ is a commutative $\bb Q$-algebra and $A$ is a $k$-algebra, then the canonical map $\HP_n^k(A[T])\to\HP_n^k(A)$ is an isomorphism for all $n\in\bb Z$. Indeed, letting $P_\blob\to A$ be a simplicial resolution of $A$ by free $k$-algebras, it follows from the proof of the homotopy invariance of usual periodic cyclic homology that the periodicity map $S:\HC_{n+2}^k(P_p[T],\pid T)\to\HC_n^k(P_p[T],\pid T)$ on relative cyclic homology is zero for all $n,p\ge0$ \cite[\S4.1.12]{Loday1992}; since the simplicial complex $\CC_\blob^k(P_\blob[T],\pid T)$ computes $\HC^k_*(A[T],\pid T)$ it follows that $S:\HC_{n+2}^k(A[T],\pid T)\to\HC_n^k(A[T],\pid T)$ is also zero, as required.
\end{remark}

\section{Andr\'e--Quillen homology}\label{section_AR_for_AQ}
The remainder of the paper is devoted to extending the pro excision, Artin--Rees, and continuity results which we have already proved for Hochschild homology to Andr\'e--Quillen homology. This is technically more difficult because the restriction spectral sequence of Proposition \ref{proposition_HH_SS} must be replaced by the higher Jacobi--Zariski spectral sequence of C.~Kassel and A.~Sletsj\o e (see Proposition \ref{proposition_JZSS} and Remark \ref{remark_JZSS}), which is unfortunately more cumbersome to use.

We begin with a review of Andr\'e--Quillen homology \cite{Andre1974, Quillen1970, Ronco1993}, though we assume the reader is familiar with its basic properties. {\em All rings and simplicial rings in this section are commutative.} Let $k\to A$ be a homomorphism of rings; let $P_\bullet\to A$ be a simplicial resolution of $A$ by free (commutative!) $k$-algebras, and set \[\bb L_{A|k}:= \Omega_{P_\bullet|k}^1\otimes_{P_\bullet}A.\] Thus $\bb L_{A|k}$ is a simplicial $A$-module which is free in each degree; it is called the {\em cotangent complex} (though we always consider it simplicially) of the $k$-algebra $A$. The cotangent complex is well-defined up to homotopy, since the free simplicial resolution $P_\bullet\to A$ is unique up to homotopy.

Set $\bb L_{A|k}^i:=\bigwedge_A^i\bb L_{A|k}$ for each $i\ge 1$. The {\em Andr\'e--Quillen homology} of the $k$-algebra $A$, with coefficients in any $A$-module $M$, is defined by \[D_n^i(A|k,M):=\pi_n(\bb L_{A|k}^i\otimes_A M),\] for $n\ge 0$, $i\ge 1$. When $M=A$ the notation is simplified to \[D_n^i(A|k):=D_n^i(A|k,A)=\pi_n(\bb L_{A|k}^i).\] When $i=1$ the superscript is often omitted, writing $D_n(A|k,M)=\pi_n(\bb L_{A|k}\otimes_A M)$ and $D_n(A|k)=\pi_n(\bb L_{A|k})$ instead. If $k\to A$ is essentially of finite type and $k$ is Noetherian, then $D_n^i(A|k,M)$ is a finitely generated $A$-module for all $n,i$ and for all finitely generated $A$-modules $M$.

If $0\to M\to N\to P\to 0$ is a short exact sequence of $A$-modules, then there is a resulting long exact sequence for each $i\ge 1$: \[\cdots\To D_n^i(A|k,M)\To D_n^i(A|k,N)\To D_n^i(A|k,P)\To \cdots\]

Finally, if $k\to A\to B$ are homomorphisms of rings then the simplicial resolutions may be chosen so that there is an exact sequence of simplicial $B$-modules \[0\To\bb L_{A/k}\otimes_A B\To\bb L_{B/k}\To\bb L_{B/A}\To 0.\] This remains exact upon tensoring by any $B$-module $M$ since these simplicial $B$-modules are free in each degree; taking homotopy yields the Jacobi--Zariski long exact sequence of $B$-modules \[\cdots\To D_n(A|k,M|_A)\To D_n(B|k,M)\To D_n(B|A,M)\To\cdots\]

\begin{remark}
To avoid any ambiguity once spectral sequences appears, we remark that the notation $D_n^i(A|k,M)$ is defined in the same way if $i\le 0$ or $n<0$. However, $D_n^i(A|k,M)=0$ if $n<0$ and \[D_n^0(A|k,M)=\begin{cases}M&n=0,\\0&\mbox{else,}\end{cases}\] since $\bb L_{A|k}^0\otimes_AM\simeq M$.
\end{remark}

\begin{remark}\label{remark_AQ_to_HH}
Let $k\to A$ be a homomorphism of rings, and $M$ an $A$-module. Then there is a first quadrant spectral sequence of $A$-modules \[E^2_{pq}=D_p^q(A|k,M)\Longrightarrow H_{p+q}^k(A,M).\]
Indeed, letting $P_\blob\to A$ be a free simplicial resolution, this is the spectral sequence associated to the bisimplicial $A$-module $C_\bullet^k(P_\bullet,M)$, since $\pi_*(C_\blob^k(P_p,M))=H^k_*(P_p,M)\cong\Omega_{P_p/k}^*\otimes_{P_p}M$ for each $p\ge 0$ (the final isomorphism follows from the calculation of Hochschild homology for free commutative algebras \cite[Thm.~3.2.2]{Loday1992}). We will only use this spectral sequence in Lemma \ref{lemma_AQ_vanishing_in_excision}.
\end{remark}

\begin{remark}[Sketch of proof of Theorem \ref{theorem_pro_excision}]\label{remark_sketch_proof}
It may be instructive at this point to give a sketch of the upcoming proof of Theorem \ref{theorem_pro_excision}, namely that pro Tor-unital ideals satisfy pro excision in Andr\'e--Quillen homology under a minor assumption.

Let $k\to A\to B$ be homomorphisms of commutative rings, and $I$ a pro Tor-untial ideal of $A$ mapped isomorphically to an ideal of $B$. It is required to show that, for each $i\ge 0$, the square of pro cotangent complexes 
\[\xymatrix{
\bb L^i_{A|k}\ar[r]\ar[d]&\bb L^i_{B|k}\ar[d]\\
\{\bb L^i_{A/I^r|k}\}_r\ar[r]&\{\bb L^i_{B/I^r|k}\}_r
}\]
is homotopy cartesian. The first step is to show that $\{\bb L_{A/I^r|k}^i\}_r\simeq\{\bb L^i_{A|k}\otimes_AA/I^r\}_r$, i.e., that
\[\{D_n^i(A/I^r|k)\}_r\cong\{D_n^i(A|k,A/I^r)\}_r\tag{1}\] for all $n\ge0$ (and similarly for $B$). Once this is achieved, it follows that the homotopy fibres of the vertical arrows in the above square are $\{\bb L^i_{A|k}\otimes_AI^r\}_r$ and $\{\bb L^i_{B|k}\otimes_BI^r\}_r$; we must show these are weakly equivalent. This will follow from the vanishing result \[\{\bb L^i_{B|A}\otimes_BI^r\}\simeq 0\tag{2}\] and the Jacobi--Zariski style implication \[\{\bb L^i_{B/A}\otimes_BI^r\}\simeq 0\Longrightarrow \{\bb L^i_{A|k}\otimes_AI^r\}\simeq\{\bb L^i_{B|k}\otimes_BI^r\}_r.\tag{3}\]

Results of flavour (1) -- (3) are the Artin--Rees and continuity properties which will be established in Theorem \ref{theorem_AR_properties_in_AQ_homology} and Lemma \ref{lemma_excision_1}, after which the above programme to prove pro excision can be carried out.
\end{remark}

\subsection{Continuity properties for Andr\'e--Quillen homology}\label{subsection_Artin_Rees_for_AQ}
Our goal in this section is to prove Artin--Rees and continuity properties for Andr\'e--Quillen homology. The $i=1$ case of the next theorem was proved by Andr\'e and Quillen whenever $A$ is Noetherian \cite[Prop.~X.12]{Andre1974} \cite[Thm.~6.15]{Quillen1970}; we generalise their result to pro Tor-unital ideals, and moreover to the case $i>1$. Krishna \cite{Krishna2010} proved Corollary \ref{corollary_AR_properties_in_AQ_homology} for certain cases of finitely generated algebras over characteristic zero fields, and we were motivated by his calculations.

\begin{theorem}\label{theorem_AR_properties_in_AQ_homology}
Let $k\to A$ be a homomorphism of rings, $I$ a pro Tor-unital ideal of $A$, and $M$ an $A$-module. Then:
\begin{enumerate}
\item $\{D_n^i(A/I^r|A,M/I^rM)\}_r=0$ for all $n\ge0$, $i\ge1$.
\item The canonical map $\{D_n^i(A|k,M/I^rM)\}_r\To\{D_n^i(A/I^r|k,M/I^rM)\}_r$ is an isomorphism for all $n\ge 0$, $i\ge0$.
\end{enumerate}
\end{theorem}

In particular, setting $M=A$ we immediately obtain the following corollary, which is the analogue for Andr\'e--Quillen homology of Theorem \ref{theorem_AR_for_HH}:

\begin{corollary}\label{corollary_AR_properties_in_AQ_homology}
Let $k\to A$ be a homomorphism of commutative rings, and $I$ a pro Tor-unital ideal of $A$. Then:
\begin{enumerate}
\item $\{D_n^i(A/I^r|A)\}_r=0$ for all $n\ge0$, $i\ge1$.
\item The canonical map $\{D_n^i(A|k,A/I^r)\}_r\To\{D_n^i(A/I^r|k)\}_r$ is an isomorphism for all $n\ge 0$, $i\ge 0$.
\end{enumerate}
\end{corollary}

\begin{remark}
The previous theorem and corollary may be stated directly in terms of the pro cotangent complexes. Given a projective system $X_\bullet(1)\leftarrow X_\bullet(2)\leftarrow\cdots$ of simplicial $A$-modules, we may take the limit degree-wise to form $\{X_\bullet(r)\}_r$, which is a simplicial object in the abelian category of pro $A$-modules. By construction, its homotopy pro groups are the pro $A$-modules $\pi_n(\{X_\bullet(r)\}_r):=\{\pi_n(X_\bullet(r))\}_r$.
The notation suggests that $\{X_\bullet(r)\}_r$ lives in the abelian category $\op{Pro}(A\op-mod^{\Delta^\sub{op}})$, i.e.~pro objects in the category of simplicial $A$-modules, but we prefer to view it in $\op{Pro}(A\op-mod)^{\Delta^\sub{op}}$, i.e.~simplicial objects in the category of pro $A$-modules, via the natural functor \[\op{Pro}(A\op-mod^{\Delta^\sub{op}})\To\op{Pro}(A\op-mod)^{\Delta^\sub{op}}.\] We do this because terms like ``acyclic'' are already defined in $\op{Pro}(A\op-mod)^{\Delta^\sub{op}}$; otherwise we would have to introduce a suitable model structure on $\op{Pro}(A\op-mod^{\Delta^\sub{op}})$.

The statements of Theorem \ref{theorem_AR_properties_in_AQ_homology} are equivalent to the following:
\begin{enumerate}
\item $\{\bb L_{A/I^r|A}^i\otimes_{A/I^r}M/I^r\}_r$ is acyclic.
\item $\{\bb L_{A|k}^i\otimes_AM/I^rM\}_r\To\{\bb L_{A/I^r|k}^i\otimes_{A/I^r}M/I^rM\}_r$ is a weak equivalence.
\end{enumerate}
We will take advantage of this pro simplicial framework when we turn to excision in Section \ref{section_AQ_excision}; it is not essential but simplifies the exposition there.
\end{remark}

We now turn to the proof of Theorem \ref{theorem_AR_properties_in_AQ_homology}. We begin by treating part (i), which is based on Andr\'e's original work:

\begin{proof}[\bf Proof of part (i) of Theorem \ref{theorem_AR_properties_in_AQ_homology}]
Let $I\subseteq A$ be a pro Tor-unital ideal and fix $i\ge1$. We first claim that in order to prove \[\{D_n^i(A/I^r|A,M/I^rM)\}_r=0\] for all $A$-modules $M$ and all $n\ge0$, it is sufficient to consider the case that $M$ is an $A/I$-module. Indeed, once the vanishing claim has been proved for $A/I$-modules, it immediately follows for $A/I^r$-modules by induction on $r\ge1$ using the short exact sequence $0\to I^{r-1}M\to M\to M/I^{r-1}M\to 0$. Then, if $M$ is an arbitrary $A$-module and $r\ge1$ is given, we apply the special case to the $A/I^r$-module $M/I^rM$ to find $s\ge r$ such that the second of the following arrows, hence the composition, is zero:
\[D_n^i(A/I^s|A,M/I^sM)\To D_n^i(A/I^s|A,M/I^rM)\To D_n^i(A/I^r|A,M/I^rM).\] This shows that $\{D_n^i(A/I^r|A,M/I^rM)\}_r=0$, as desired.

Now let $M$ be an $A/I$-module; we must show that $\{D_n^i(A/I^r|A,M)\}_r=0$. Firstly, Corollary \ref{corollary_most_useful_pro_vanishing} implies that $\{\Tor_n^A(A/I^r,A/I)\}_r=0$ for all $n\ge1$. Fixing $n,r\ge1$, we may therefore find a sequence of integers $r_0\ge r_1\ge\cdots\ge r_n:=r$ such that the maps \[\Tor_p^A(A/I^{r_{p-1}},A/I)\To\Tor_p^A(A/I^{r_p},A/I)\] are zero for $p=1,\dots,n$.

For each $s\ge1$, let $P_\bullet(s)\to A/I^s$ be a functorially chosen simplicial resolution of $A/I^s$ by free $A$-algebras. Verbatim following Andr\'e \cite[Proof of Prop.~X.12]{Andre1974} (with notation $A^p=A$, $B^p=A/I^{r_p}$, $C^p=A/I$, and $W=M$), we deduce that there exist a simplicial resolution $X_\bullet\to A/I$ by free $A/I$-algebras (!), and a simplicial subring $F_\bullet\subseteq X_\bullet$, with the following properties (of which we will only need (i) and (v)):
\begin{enumerate}\itemsep-1pt
\item There exists a factoring $P_\bullet(r_0)\otimes_AA/I\to F_\bullet \to P_\bullet(r_n)\otimes_AA/I$.
\item Under the above maps, $F_p$ is a free $P_p(r_0)\otimes_AA/I$-algebra for all $p\ge0$.
\item $\pi_0(P_p(r_0)\otimes_AA/I)\to \pi_0(F_\bullet)$ is an isomorphism.
\item $\pi_p(F_\bullet)=0$ for $p=1,\dots,n$.
\item $X_p=F_p$ for $p=0,\dots,n+1$.
\end{enumerate}
Since $X_\bullet$ is a simplicial resolution of $A/I$ by free $A/I$-algebras, we have \[\pi_p(\Omega^i_{X_\bullet|A/I}\otimes_{X_\blob}M)=D_p^i(A/I|A/I,M)=0\] for all $p\ge0$, in particular for $p=n$. Since $X_p=F_p$ for $p=n-1,n,n+1$, it follows that $H_n(\Omega^i_{F_\bullet|A/I}\otimes_{F_\blob}M)=0$. According to property (i), the map \[D_n^i(A/I^{r_0}|A,M)=\pi_n(\Omega^i_{P_\bullet(r_0)|A}\otimes_{P_\blob(r_0)}M)\to \pi_n(\Omega^i_{P_\bullet(r_n)|A}\otimes_{P_\blob(r_n)}M)=D_n^i(A/I^{r_n}|A,M)\] is therefore zero.

So, given any integer $r\ge1$ we have found an integer $r_0\ge r$ for which the map $D_n^i(A/I^{r_0}|A,M)\to D_n^i(A/I^r|A,M)$ is zero, as required.
\end{proof}

It remains to prove part (ii) Theorem \ref{theorem_AR_properties_in_AQ_homology}; when $i=1$ this is easily deduced from the usual Jacobi--Zariski long exact sequence:

\begin{proof}[\bf Proof of part (ii) of Theorem \ref{theorem_AR_properties_in_AQ_homology} when $\boldsymbol{i=1}$]
Let $k\to A$ be a homomorphism of rings, $I\subseteq A$ a pro Tor-unital ideal, $M$ an $A$-module, and fix $i\ge 1$. Then the long exact Jacobi--Zariski sequence for $k\to A\to A/I^r$ is \[\cdots\To D_n(A|k,M/I^rM)\To D_n(A/I^r|k, M/I^rM)\To D_n(A/I^r|A,M/I^rM)\To\cdots\] Part (i) of Theorem \ref{theorem_AR_properties_in_AQ_homology} shows that the limit over $r$ of the right-most term is zero; so taking the limit over $r$ gives $\{D_n(A|k,M/I^rM)\}_r\isoto\{D_n(A/I^r|k, M/I^rM)\}_r$, proving (ii).
\end{proof}

To prove part (ii) of Theorem \ref{theorem_AR_properties_in_AQ_homology} when $i>1$ we will use the following ``higher Jacobi--Zariski'' spectral sequence, which is the analogue in Andr\'e--Quillen homology of Proposition \ref{proposition_HH_SS}:

\begin{proposition}[Kassel--Sletsj\o e \cite{Kassel1992}]\label{proposition_JZSS}
Let $k\to A\to B$ be homomorphisms of rings, let $M$ be a $B$-module, and fix $i\ge 1$. Then there is a natural, third octant, bounded spectral sequence of $B$-modules \[E^1_{pq}\Longrightarrow D_{p+q}^i(B|k,M)\] whose columns may be described as follows:
\begin{itemize}
\item Suppose $p<-i$ or $p>0$. Then $E^1_{pq}=0$.
\item Suppose $p=-i$; i.e., we are on the left-most column of the $E^1$-page. Then $E^1_{pq}=D_{q-i}^i(A|k,M)$.
\item Suppose $-i<p\le0$. Then the $p^\sub{th}$ column of the $E^1$-page is given by a first quadrant spectral sequence \[\cal E^2_{\al\beta}=D_\al^{-p}\big(A|k,D_\beta^{i+p}(B|A,M)\big)\Longrightarrow E_{p,\al+\beta-p}^1.\]
\end{itemize}
\end{proposition}

\begin{remark}\label{remark_JZSS}
Suppose that $k\to A\to B$, $M$, $i$ are as in Proposition \ref{proposition_JZSS}, and assume that $A$ is filtered inductive limit of smooth, finite-type $k$-algebras. Then $\bb L_{A|k}^{-p}\simeq\Omega_{A|k}^{-p}$ for all $p$ by the ``Hochschild--Kostant--Rosenberg theorem for Andr\'e--Quillen homology'' \cite[Thm.~3.5.6]{Loday1992}, which are flat $A$-modules, so the spectral sequence simplifies to \[E_{pq}^1=\Omega_{A|k}^{-p}\otimes_AD_{p+q}^{i+p}(B|A,M)\Longrightarrow D_{p+q}^i(B|k,M).\]
It is this simplified form of Kassel--Sletsj\o e's spectral sequence which most often appears in applications, e.g., \cite{Cortinas2009}, but we require it in full generality.
\end{remark}

Using the higher Jacobi--Zariski spectral sequence we may complete the proof of Theorem \ref{theorem_AR_properties_in_AQ_homology}:

\begin{proof}[\bf Proof of part (ii) of Theorem \ref{theorem_AR_properties_in_AQ_homology} for $\pmb{i\ge1}$]
Let $k\to A$ be a homomorphism of rings, $I\subseteq A$ a pro Tor-unital ideal, $M$ an $A$-module, and fix $i\ge 1$.

For each $r\ge 1$ we apply Proposition \ref{proposition_JZSS} to the ring homomorphisms $k\to A\to A/I^r$ and the $A/I^r$-modules $M/I^rM$ to obtain natural, third octant, bounded spectral sequences \[E^1_{pq}(r)\Longrightarrow D_{p+q}^i(A/I^r|k,M/I^rM),\tag{\dag}\] with the properties described by the proposition.

We claim that $\{E_{pq}^1(r)\}_r=0$ if $p>-i$. It is sufficient to consider the situation $-i<p\le 0$, in which case there are natural, first quadrant spectral sequences for all $r\ge 1$: \[\cal E_{\al\beta}^2(r)=D_\al^{-p}\big(A|k,D_\beta^{i+p}(A/I^r|A,M/I^rM)\big)\Longrightarrow E_{p,\al+\beta-p}^1.\] Passing to the limit yields a spectral sequence of pro $A$-modules, \[\cal E_{\al\beta}^2(\infty)=\{\cal E_{\al\beta}^2(r)\}_r\Longrightarrow \{E_{p,\al+\beta-p}^1(r)\}_r,\] and so to prove our vanishing claim it is enough to show that $\cal E_{\al\beta}^2(\infty)=0$ for all $\al,\beta$. But part (i) of Theorem \ref{theorem_AR_properties_in_AQ_homology} (which we have already proved) implies that for any $r$ there exists $s\ge r$ such that the map \[\cal E_{\al\beta}^2(s)=D_\beta^{i+p}(A/I^s|A,M/I^sM)\To D_\beta^{i+p}(A/I^r|A,M/I^rM)=\cal E_{\al\beta}^2(r)\] is zero; hence $\cal E_{\al\beta}^2(\infty)=0$, completing the proof that $\{E_{pq}^1(r)\}_r=0$ for $p>-i$.

We now form the limit of the spectral sequences (\dag): \[E_{pq}^1(\infty)=\{E_{pq}^1(r)\}_r\Longrightarrow \{D_{p+q}^i(A/I^r|k,M/I^rM)\}_r\] By what we have just proved, and from Proposition \ref{proposition_JZSS}, this spectral sequence is everywhere zero on the first page except along the column $p=-i$, where it equals \[E_{-i,q}^1(\infty)=\{D_{q-i}^i(A|k,M/I^rM)\}_r.\] Therefore the edge map $\{D_{n}^i(A|k,M/I^rM)\}_r\to\{D_n^i(A/I^r|k,M/I^rM)\}_r$ is an isomorphism, as required to complete the proof of Theorem \ref{theorem_AR_properties_in_AQ_homology}.
\end{proof}

\subsection{Pro excision for Andr\'e--Quillen homology}\label{section_AQ_excision}
In this section we use the Artin--Rees properties of Theorem \ref{theorem_AR_properties_in_AQ_homology} to prove that pro Tor-unital ideals satisfy pro excision in Andr\'e--Quillen homology. Unfortunately we must impose the following additional condition on the ideal, though we suspect it is unnecessary:

\begin{definition}\label{definition_small}
We will say that a non-unital (commutative) ring $I$ is {\em small} if and only if for each $r\ge 1$ there exists $s\ge 1$ such that $I^{(r)}\supseteq I^s$, where $I^{(r)}$ is the ideal of $I$ generated by the $r^\sub{th}$-powers of elements of $I$.
\end{definition}

\begin{example}
Suppose that $I$ is an ideal of a ring $A$. Then $I$ is small if it is a finitely generated ideal of $A$, or if $A$ is a $\bb Q$-algebra (the latter case follows from combinatorial identities such as $ab=\tfrac12((a+b)^2-a^2-b^2)$; see \cite{Anderson1994}).
\end{example}

It is the proof of the follow vanishing result in which we require smallness:

\begin{lemma}\label{lemma_AQ_vanishing_in_excision}
Let $A\to B$ be a homomorphism of rings, and $I$ an ideal of $A$ mapped isomorphically to an ideal of $B$.
\begin{enumerate}
\item Then, for each $n>0$, the $A$-module $\HH_n^A(B)$ is annihilated by a power of $I$.
\item Assume further that $I$ is small; then, for each $i\ge1$, $n\ge0$, the $A$-module $D_n^i(B|A)$ is annihilated by a power of $I$.
\end{enumerate}
\end{lemma}
\begin{proof}
(i): We will say that a map of simplicial $A$-modules $M_\blob\to N_\blob$ is an {\em $I$-weak equivalence} if and only if the kernel and cokernel of $\pi_n(M_\blob)\to\pi_n(N_\blob)$ are killed by a power of $I$ for all $n\ge0$. The composition or derived tensor product of two $I$-weak equivalences is clearly again an $I$-weak equivalence.

Let $P_\blob\to B$ be a simplicial resolution of $B$ by free $A$-algebras. Since the kernel and cokernel of $A\to B$ are killed by $I$, the structure map $A\to P_\blob$ is an $I$-weak equivalence. By the comments in the previous paragraph, it inductively follows that the structure map $A\to C_\blob^A(P_\blob)$ is an $I$-weak equivalence, which is exactly the desired assertion.

(ii): We now assume further that $I$ is small, and we will prove by induction on $N\ge0$ the following: for all $i\ge1$ and $0\le n\le N$, the $A$-module $D_n^i(B|A)$ is killed by a power of $I$.

To treat the base case $N=0$ we show that $D_0^i(B|A)=\Omega_{B|A}^i$ is killed by $I$. But this is straightforward and well-known: if $b\in B$ and $a\in I$ then $a\,db=d(ab)=0$, since $ab\in I$ and hence $ab$ comes from an element of $A$.

Proceeding by induction, we may suppose that $D_n^i(B|A)$ is killed by a power of $I$ for all $i\ge 1$ and $0\le n\le N$. It therefore follows that the cokernel of the edge map $\HH_{N+2}^A(B)\to D_{N+1}^1(B|A)$ in the Andr\'e--Quillen-to-Hochschild-homology spectral sequence $E^2_{pq}=D_p^q(B|A)\Rightarrow \HH_{p+q}^A(B)$ (see Remark \ref{remark_AQ_to_HH}) is killed by a power of $I$. But part (i) implies that $\HH_{N+2}^A(B)$ is also killed by a power of $I$, whence $D_{N+1}^1(B|A)$ is killed by a power of $I$.

Therefore there exists $M>0$ such that the homotopy groups $\pi_n(\bb L_{B|A})=D_n^1(B|A)$ are killed by $I^M$, for $0\le n\le N+1$. According to Proposition \ref{proposition_homological_factoring_ideal}, it follows that $\pi_n(\bb L_{B|A}^i)=D_n^i(B|A)$ is killed by the $N+2$ power of $(I^M)^{(i)}$, which denotes the ideal generated by the $i^\sub{th}$-powers of $I^M$, for all $i\ge1$ and $0\le n\le N+1$.

The inductive step and the proof will thus be completed as soon as we show that $(I^M)^{(i)}$ contains a power of $I$. But this is a consequence of the smallness assumption, as $(I^M)^{(i)}\supseteq I^{(Mi)}$.
\end{proof}

\begin{remark}
Suppose that $A\to B$ is a homomorphism of rings, and that $I$ is a pro Tor-unital ideal of $A$ mapped isomorphically to an ideal of $B$. Applying Theorem \ref{theorem_pro_excision_for_HH_HC} with $k=A$, we easily see that $\HH_n^A(B)\cong\{\HH_n^A(B/I^r)\}$ for all $n\ge0$; this strengthens Lemma \ref{lemma_AQ_vanishing_in_excision}(i) in the pro Tor-unital case, as it implies that $\HH_n^A(B)$ embeds into $\HH_n^A(B/I^r)$ for $r\gg0$.

Continuing to assume that $I$ is pro Tor-unital, it seems likely that also $D_n^i(B|A)\cong \{D_n^i(B/I^r|A)\}_r$, even without any smallness condition. This would allow the smallness assumption to be removed from Lemma \ref{lemma_excision_1} and Theorem \ref{theorem_pro_excision}. (We should mention that, unfortunately, there do exist pro Tor-unital ideals which are not small; e.g., the quasi-regular ideal $I=\pid{X_1,X_2,\dots}$ of the infinite polynomial ring $\bb F_2[X_1,X_2,\cdots]$).
\end{remark}

\begin{lemma}\label{lemma_excision_1}
Let $k\to A\to B$ be homomorphisms of rings, and $I$ a small, pro Tor-unital ideal of $A$ mapped isomorphically to an ideal of $B$. Then the following canonical maps are isomorphisms for all $n,i\ge 0$:
\begin{enumerate}\itemsep1pt
\item $\{D_n^i(B|A)\otimes_BB/I^r\}_r\To\{D_n^i(B|A,B/I^r)\}_r$.
\item $\{D_n^i(B|A,I^r)\}_r\To\{I^rD_n^i(B|A)\}_r\stackrel{(\ast)}{=}0$\\ (vanishing $(\ast)$ not necessarily valid if $i=n=0$).
\item $\{D_n^i(A|k,I^r)\}_r\To\{D_n^i(B|k,I^r)\}_r$.
\end{enumerate}
\end{lemma}
\begin{proof}
The claims are all trivial when $i=0$, so we assume that $i>0$ throughout.

(i): The universal coefficient spectral sequence for the $B$-module $B/I^r$ is \[E_{pq}^2(r)=\Tor_p^B(D_q^i(B|A),B/I^r)\Longrightarrow D_{p+q}^i(B|A,B/I^r),\]  and taking the limit over $r$ yields a first quadrant spectral sequence of pro $A$-modules: \[\{\Tor_p^A(D_q^i(B|A),A/I^r)\}_r\Longrightarrow \{D_{p+q}^i(B|A,B/I^r)\}_r.\] By Lemma \ref{lemma_AQ_vanishing_in_excision}(ii), $D_q^i(B|A)$ is killed by a power of $I$, so Corollary \ref{corollary_most_useful_pro_vanishing} implies that $\{\Tor_p^A(D_q^i(B|A),A/I^r)\}_r=0$ for $p>0$; thus the spectral sequence collapses to edge map isomorphisms $\{D_n^i(B|A)\otimes_BB/I^r\}_r\isoto\{D_n^i(B|A,B/I^r)\}_r$, as desired.

(ii): The short exact sequences $0\to I^r\to B\to B/I^r\to 0$ induce a long exact sequence of pro $A$-modules: \[\cdots\To \{D_n^i(B|A,I^r)\}_r\To D_n^i(B|A)\xto{(1)} \{D_n^i(B|A,B/I^r)\}_r\To\cdots\] By part (i), the third term may be identified with $\{D_n^i(B|A)\otimes_BB/I^r\}_r$, whence arrow (1) is surjective, the sequence breaks into short exact sequences, and the canonical map $\{D_n^i(B|A,I^r)\}_r\to\{I^rD_n^i(B|A)\}_r$ is an isomorphism. The vanishing claim ($\ast$) follows from Lemma \ref{lemma_AQ_vanishing_in_excision}(ii).

(iii): The proof of (iii) is based on the higher Jacobi--Zariski spectral sequence of Proposition \ref{proposition_JZSS}. It implies that for each $r\ge 1$ there is a natural, third quadrant, bounded, spectral sequence \[E_{pq}^1(r)\Longrightarrow D_{p+q}^i(B|k,I^r)\] which vanishes outside $-i\le p\le 0$, whose $-i^\sub{th}$ column is given by $E_{-i,q}^1=D_{q-i}^i(A|k,I^r)$, and which is described in the range $-i<p\le 0$ by natural, first quadrant spectral sequences \[\cal E_{\al\beta}^2(r)=D_\al^{-p}(A|k,D_\beta^{i+p}(B|A,I^r))\Longrightarrow E_{p,\al+\beta-p}^1(r).\] According to part (ii), $\{\cal E_{\al\beta}^2(r)\}_r=0$, whence the limit of the $E$-spectral sequences collapses to edge map isomorphisms $\{D_n^i(A|k,I^r)\}_r\isoto\{D_n^i(B|k,I^r)\}_r$.
\end{proof}

We may now prove the main pro excision result for Andr\'e--Quillen homology, following the sketch given in Remark \ref{remark_sketch_proof}:

\begin{theorem}[Pro excision for AQ homology]\label{theorem_pro_excision}
Let $k\to A\to B$ be homomorphisms of rings, and $I$ a small, pro Tor-unital ideal of $A$ mapped isomorphically to an ideal of $B$. Then the following square of simplicial pro $A$-modules is homotopy cartesian
\[\xymatrix{
\bb L^i_{A|k}\ar[r]\ar[d]&\bb L^i_{B|k}\ar[d]\\
\{\bb L^i_{A/I^r|k}\}_r\ar[r]&\{\bb L^i_{B/I^r|k}\}_r
}\]
thereby resulting in a long exact, Mayer--Vietoris sequence of pro $A$-modules \[\cdots\To D_n^i(A|k)\To\{D_n^i(A/I^r|k)\}_r\oplus D_n^i(B|k)\To\{D_n^i(B/I^r|k)\}_r\To\cdots\]
\end{theorem}
\begin{proof}
According to Corollary \ref{corollary_AR_properties_in_AQ_homology}(ii), which applies to both $A$ and $B$, the homotopy of the simplicial modules in the bottom row are unchanged if we replace them by \[\{\bb L_{A|k}^i\otimes_AA/I^r\}_r\To\{\bb L_{B|k}^i\otimes_BB/I^r\}_r.\] The vertical arrows in the square now become surjective, and so to prove that the square is homotopy cartesian it is enough to show that the induced map on the vertical kernels, namely \[\{\bb L_{A|k}^i\otimes_AI^r\}_r\To\{\bb L_{B|k}^i\otimes_BI^r\}_r\]
is a weak equivalence. But this is precisely Lemma \ref{lemma_excision_1}(iii).
\end{proof}


\begin{appendix}
\section{Arguments with pro abelian groups}\label{appendix_pro}
The purpose of this appendix is to summarise the notation of pro objects, and to explain some standard arguments we use. It is likely to be unnecessary to readers already familiar with infinitesimal arguments in $K$-theory or cyclic homology.

Everything we need about categories of pro objects may be found in one of the standard references, such as the appendix to \cite{ArtinMazur1969}, or \cite{Isaksen2002}. We will often use $\op{Pro}(A\op{-}mod)$, the category of pro $A$-modules for some ring $A$, and $\op{Pro}Ab$, the category of pro abelian groups, and $\op{Pro}Rings$, the category of pro rings.

Let $\cal C$ be a category. In this paper, an object of $\op{Pro}\cal C$ is simply an inverse system $\cdots\to A_2\to A_1$ of objects and morphisms in $\cal C$, which is denoted $\{A_r\}_r$ or very occasionally $A_\infty$; some authors, including myself previously, write $\projlimf_{\!\!r}A_r$, but we now eschew this notation. Morphisms in $\op{Pro}\cal C$ are given by the rule \[\Hom_{\op{Pro}\cal C}(\{A_r\}_r,\{B_s\}_s):=\projlim_s\indlim_r\Hom_{\cal C}(A_r,B_s),\] where the right side is a genuine pro-ind limit in the category of sets, and composition is defined in the obvious way. For example, a pro object $\{A_r\}_r$ is isomorphic to zero (assuming that a zero object exists in $\cal C$) if and only if for each $r\ge 1$ there exists $s\ge r$ such that the transition map $A_s\to A_r$ is zero. A morphism is said to be {\em strict} it if arises from a compatible family of morphisms $A_r\to B_r$, for $r\ge1$.

There is a fully faithful embedding $\cal C\to\op{Pro}\cal C$. Assuming $\cal C$ has countable inverse limits, this has a right adjoint \[\op{Pro}\cal C\To \cal C,\quad \{A_r\}_r\Mapsto \projlim_r A_r,\] which is left exact but not right exact. Moreover, if $\cal C$ is an abelian category then so is $\op{Pro}\cal C$: given a inverse system of exact sequences \[\cdots\To A_{n-1}(r)\To A_n(r)\To A_{n+1}(r)\To\cdots,\] the ``limit as $r\to\infty$'', namely \[\cdots \To \{A_{n-1}(r)\}_r\To \{A_n(r)\}_r\To \{A_{n+1}(r)\}_r\To\cdots,\] is an exact sequence in $\op{Pro}\cal C$.

Pro spectral sequences play an important role in the paper, which we will discuss for concreteness only in the case of abelian groups. Suppose that \[E^1_{pq}(r)\Longrightarrow H_{p+q}(r),\] for $r\ge 1$, are spectral sequences of abelian groups, which are functorial in that we have morphisms of spectral sequences $\cdots\to E^\bullet_{pq}(2)\to E^\bullet_{pq}(1)$. To avoid convergence issues, suppose that each spectral sequence is bounded, by a bound independent of $r$; e.g., each spectral sequence might be zero outside the first quadrant. Finally, make the following assumption:
\begin{quote}
For each $p,q$ and each $r\ge 1$, there exists $s\ge r$ such that $E^1_{pq}(s)\to E^1_{pq}(r)$ is zero.
\end{quote}
Then we claim that for each $n$ and each $r\ge 1$, there exists $s\ge r$ such that $H_n(s)\to H_n(r)$ is zero. We offer two proofs of this claim; similar arguments are used in the paper:

{\em Careful proof:} For simplicity of notation, assume that each spectral sequence is zero outside the first quadrant, and let \[H_n(r)=F_nH_n(r)\supseteq\cdots\supseteq F_{-1}H_n(r)=0\] denote the resulting filtration on each $H_n(r)$. By the natural dependence of the spectral sequences on $r$, the natural maps $H_n(s)\to H_n(r)$, for $s\ge r$, respect the filtrations and thus induce homomorphisms $\op{gr}_pH_n(s)\to\op{gr}_pH_n(r)$ for $p=0,\dots,n$. But $\op{gr}_pH_n(r)$ is a subquotient of $E^1_{pq}(r)$, and the same applies to for $s$, so our standing assumption implies that we may pick $s\ge r$ such that $\op{gr}_pH_n(s)\to\op{gr}_pH_n(r)$ is zero.

So, starting with any $r\ge 1$, successively pick integers $r_n\ge \cdots\ge r_0=r$ such that  $\op{gr}_{n-i}H_n(r_i)\to\op{gr}_{n-i}H_n(r_{i-1})$ is zero for $i=1,\dots,n$. In other words, the natural map $H_n(r_i)\to H_n(r_{i-1})$ carries $F_{n-i}H_n(r_i)$ to $F_{n-i-1}H_n(r_{i-1})$. Hence the composite \[H_n(r_n)\to H_n(r_{n-1})\to\cdots \to H_n(r_1)\to H_n(r)\] carries $F_nH_n(r_n)=H_n(r_n)$ to $F_{-1}H_n(r)=0$. So setting $s=r_n$ completes the proof of the claim.

{\em Slick proof:} The naturality of the family of spectral sequences and the exactness of $\{-\}_r$ implies that we obtain a first quadrant spectral sequence of pro abelian groups \[E_{pq}^1(\infty):=\{E^1_{pq}(r)\}\Longrightarrow \{H_{p+q}(r)\}_r.\] We will often refer to this construction of this spectral sequence of pro abelian groups as ``letting $r\to\infty$''. Our assumption, rephrased into the pro language, is exactly that $E^1_{pq}(\infty)=0$ for each $p,q$.  Hence $\{H_n(r)\}_r=0$ for each $n$, which is exactly the claim.

In our calculations we will systemically make use of pro arguments of this type, usually without explicit mention; moreover, we state our results always in terms of pro objects, rather than ``for each $r\ge 1$ there exists $s\ge r$ such that etc.", whenever it does not cause confusion.

\section{A vanishing result in homological algebra}
Cell attachment arguments in homological algebra allow the following type of assertion to be established \cite[Prop.~7.3 \& Lem.~9.8]{Quillen1968}: if a map $M_\blob\to N_\blob$ of degree-wise-projective simplicial $A$-modules induces an isomorphism on homotopy in low degrees, then the same is true of $\Phi(M_\blob)\to\Phi(N_\blob)$, where $\Phi:A\op-mod\to A\op-mod$ is any functor. The aim of this appendix is to prove an analogous result in which, instead of assuming that certain groups vanish, we will assume that they are annihilated by a given ideal of $A$.

More precisely, let $A$ be commutative ring and $\Phi:A\op-mod\to A\op-mod$ a functor with the following property: there exists $e\ge1$ such that $\Phi(\mu_M(a))=\mu_{\Phi(M)}(a^e)$ for all $A$-modules $M$ and all $a\in A$, where $\mu_M(a):=$ `multiplication by $a$'$:M\to M$. Examples of such functors to have in mind are symmetric and exterior powers \[\Phi(M):=\op{Sym}_A^eM\quad\mbox{and}\quad \Phi(M):={\bigwedge}_A^eM.\] 
Given an ideal $I$ of $A$, let $I^{(e)}$ denote the ideal generated by $a^e$, for $a\in I$. Then the following is the main aim of this appendix; it is needed in the proof of Lemma \ref{lemma_AQ_vanishing_in_excision}:

\begin{proposition}\label{proposition_homological_factoring_ideal}
Let $A$ be a commutative ring and $\Phi:A\op-mod\to A\op-mod$ a functor with the above property for some $e\ge1$. Let $M_\blob$ be a simplicial $A$-module which is degree-wise projective, $k\ge1$, and $I$ an ideal of $A$ such that the homotopy groups $\pi_i(M_\blob)$, for $0\le i< k$, are annihilated by $I$. Then the homotopy groups $\pi_i(\Phi(M_\blob))$, for $0\le i< k$, are annihilated by the $k^\sub{th}$-power of the ideal $I^{(e)}$.
\end{proposition}

\begin{proof}
Let $a_0,\dots,a_{k-1}\in I$ and put $a=a_{k-1}\cdots a_0$. Then the `multiplication by $a$' map $\mu(a):M_\blob\to M_\blob$ factors as \[M_\blob\xto{\mu(a_0)}M_\blob\xto{\mu(a_1)}\cdots \xto{\mu(a_{k-1})}M_\blob,\] and our assumptions imply that $\mu(a_i)$ induces the zero map on homotopy for $0\le i< k$. Applying Proposition \ref{proposition_homological_factoring} below (via the Dold--Kan correspondence), we deduce that $\mu(a)$ may be factored as $M_\blob\to P_\blob\to M_\blob$, where $P_\blob$ is a degree-wise-projective simplicial $A$-module satisfying $\pi_i(P_\blob)=0$ for $0\le i<k$.

Since $P_\blob$ is a degree-wise-projective simplicial $A$-module whose homotopy vanishes in degrees $<k$, it is automatically contractible in degrees $<k$. That is, there exist maps $h_i:P_i\to P_{i+1}$, for $0\le i<k$, satisfying the usual identities: $d_{i+1} h_i=\op{id}_{P_i}$ and $d_jh_i=h_{i-1}d_j:P_i\to P_i$ for $0\le j\le i<k$ (with the convention that $h_{-1}d_0:=0$ on $P_0$). Applying $\Phi$ preserves these identities, and so $\Phi(P_\blob)$ is also contractible, hence acyclic, in degrees $<k$. It follows that the map \[\mu(a^e)=\Phi(\mu(a)):\pi_i(\Phi(M_\blob))\to \pi_i(\Phi(P_\blob))\to \pi_i(\Phi(M_\blob))\] is zero for $0\le i<k$, which completes the proof.
\end{proof}

The key step in the proof of Proposition \ref{proposition_homological_factoring_ideal} is the following cell attachment result in homological algebra; our proof will very closely follow an analogue for simplicial rings \cite[Lem.~X.6]{Andre1974} which was a key tool in the proof of Theorem \ref{theorem_AR_properties_in_AQ_homology}(i). For simplicity we work in the language of complexes, even though the result was used simplicially via the Dold--Kan correspondence. Henceforth in this appendix $A$ is a ring (not necessarily commutative) and we write ``module'' for ``left module'', and ``complex'' for ``chain complex supported in degrees $\ge0$''.

\begin{proposition}\label{proposition_homological_factoring}
Let $k\ge1$ be a fixed integer, and let \[M_\bullet^{(0)}\xto{f^{(0)}}M_\bullet^{(1)}\xto{f^{(1)}}\cdots\xto{f^{(k-1)}}M_\bullet^{(k)}\] be a sequence of maps of complexes of projective $A$-modules such that each map \[H_i(f^{(i)}):H_i(M_\blob^{(i)})\to H_i(M_\blob^{(i+1)})\quad (0\le i< k)\] is zero. Then the composition $f=f^{(k-1)}\circ\cdots\circ f^{(0)}$ factors as \[M_\bullet^{(0)}\to P_\blob\to M^{(k)}_\blob,\] where $P_\blob$ is a complex of projective $A$-modules satisfying $H_i(P_\blob)=0$ for $0\le i<k$.
\end{proposition}

We prove Proposition \ref{proposition_homological_factoring} via several lemmas. Given a complex $M_\blob$, its truncation in degrees $\le m$ is denoted by $M_{\le m}$.
 
\begin{lemma}\label{lemma_homological_factoring_1}
Let $m\ge 0$ be a fixed integer, and let $f:M_\bullet\to N_\bullet$ be a map of complexes of $A$-modules which induces zero on the $m^\sub{th}$ homology groups, where $M_\bullet$ is a complex of projective $A$-modules. Then $f$ may be factored as \[M_\bullet\to P_\blob\to N_\blob,\] where $P_\blob$ is a complex of projective $A$-modules satisfying:
\begin{enumerate}
\item $H_m(P_\blob)=0$.
\item $M_{\le m}\to P_{\le m}$ is an isomorphism.
\end{enumerate}
\end{lemma}
\begin{proof}
Set $Z_m:=\ker(M_m\xto{d} M_{m-1})$, and let $\pi:F\onto Z_m$ be a free $A$-module surjecting onto $Z_m$. Since $H_m(f)$ is zero, $f(Z_m)$ is contained in $\op{Im}(N_{m+1}\xto{d}N_m)$, and so for each generator $e$ of $F$ we may pick a lift $h(e)$ of $\pi(e)$ to $N_{m+1}$; this constructs a map $h:F\to N_{m+1}$ such that the diagram
\[\xymatrix{
M_m \ar[d]_{f_m} & F \ar[l]_d\ar[d]^h\\
N_m&N_{m+1}\ar[l]^d
}\]
commutes.

Define $P_n=M_n$ for $n\neq m+1$, and $P_{m+1}=M_{m+1}\oplus F$; this is the desired complex of projective $A$-modules, with non-obvious maps given as follows:
\[\xymatrix@C=1.5cm{
M_\bullet\ar[d]&\cdots& M_{m} \ar[l]\ar@{=}[d] & M_{m+1} \ar[l]\ar[d]^{(\op{id},0)} & M_{m+2} \ar[l]\ar@{=}[d] & \cdots\ar[l]\\
P_\bullet\ar[d]&\cdots & M_{m} \ar[l]\ar[d] & M_{m+1}\oplus F \ar[l]_{d+\pi}\ar[d]^{f_m+h} & M_{m+2} \ar[l]_{(d,0)}\ar[d] & \cdots\ar[l]\\
N_\bullet&\cdots & N_{m} \ar[l] & N_{m+1} \ar[l] & N_{m+2} \ar[l] & \cdots\ar[l]
}\]
\end{proof}

\begin{lemma}
In addition to the hypotheses of Lemma \ref{lemma_homological_factoring_1}, assume further that there is an integer $k\ge1$ such that $H_i(N_\blob)=0$ for $m<i< m+k$. Then $f$ may be factored as \[M_\bullet\to P_\blob\to N_\blob,\] where $P_\blob$ is a complex of projective $A$-modules satisfying:
\begin{enumerate}
\item $H_i(P_\blob)=0$ for $m\le i< m+k$.
\item $M_{\le m}\to P_{\le m}$ is an isomorphism.
\end{enumerate}
\end{lemma}
\begin{proof}
The case $k=1$ is covered by the previous lemma. By induction on $k$, we may therefore assume that $f$ may be factored as $M_\bullet\to P_\blob'\xto{g} N_\blob$, where $P_\blob'$ is a complex of projective $A$-modules satisfying:
\begin{enumerate}
\item $H_i(P'_\blob)=0$ for $m\le i<m+k-1$.
\item $M_{\le m}\to P'_{\le m}$ is an isomorphism.
\end{enumerate}
Now apply the previous lemma to the map $g:P_\blob'\to N_\blob$ and integer $m+k-1$ to obtain a factoring of $g$ as $P_\blob'\to P_\blob\to N_\blob$, where $P_\blob$ is a complex of projective $A$-modules satisfying
\begin{enumerate}
\item $H_{m+k-1}(P_\blob)=0$.
\item $P'_{\le m+k-1}\to P_{\le m+k-1}$ is an isomorphism.
\end{enumerate}
The factoring of $f$ through $P_\blob$ has all the desired properties.
\end{proof}

\begin{lemma}
Let $m\ge0$, $k\ge1$ be fixed integers, and let \[M_\bullet^{(m)}\xto{f^{(m)}}M_\bullet^{(m+1)}\xto{f^{(m+1)}}\cdots\xto{f^{(m+k-1)}}M_\bullet^{(m+k)}\] be a sequence of maps of complexes of projective $A$-modules such that each map \[H_i(f^{(i)}):H_i(M_\blob^{(i)})\to H_i(M_\blob^{(i+1)})\quad (m\le i\le m+k-1)\] is zero. Then the composition $f=f^{(m+k-1)}\circ\cdots\circ f^{(m)}$ may be factored as \[M_\bullet^{(m)}\to P_\blob\to M^{(m+k)}_\blob,\] where $P_\blob$ is a complex of projective $A$-module satisfying:
\begin{enumerate}
\item $H_i(P_\blob)=0$ for $m\le i<m+k$.
\item $M_{\le m}^{(m)}\to P_{\le m}$ is an isomorphism.
\end{enumerate}
\end{lemma}
\begin{proof}
The case $k=1$ is covered by Lemma \ref{lemma_homological_factoring_1}. By induction on $k$, we may therefore suppose that the composition $f'=f^{(m+k-1)}\circ\cdots\circ f^{(m+1)}$ factors as as $M_\bullet^{(m+1)}\xto{g} P'_\blob\to M^{(m+k)}_\blob$, where $P'_\blob$ is a complex of projective $A$-module satisfying:
\begin{enumerate}
\item $H_i(P'_\blob)=0$ for $m+1\le i<m+k$.
\item $M_{\le m+1}^{(m+1)}\to P_{\le m+1}'$ is an isomorphism.
\end{enumerate}
The map $M_\blob^{(m)}\xto{f^{(m)}}M_\blob^{(m+1)}\xto{g} P_\blob'$ satisfies all the conditions of the map $f$ in the statement of the previous lemma, so we deduce that $g\circ f^{(m)}$ may be factored as $M_\blob^{(m)}\to P_\blob\to P_\blob'$, where $P_\blob$ is a complex of projective $A$-modules satisfying
\begin{enumerate}
\item $H_i(P_\blob)=0$ for $m\le i< m+k$.
\item $M_{\le m}^{(m)}\to P_{\le m}$ is an isomorphism.
\end{enumerate}
So the factoring of $f$ given by $M_\blob^{(m)}\to P_\blob\to M_\blob^{(m+k)}$ has all the desired properties.
\end{proof}

\begin{proof}[Proof of Proposition \ref{proposition_homological_factoring}]
Take $m=0$ in the previous lemma.
\end{proof}

\end{appendix}
\bibliographystyle{acm}
\bibliography{../Bibliography}

\noindent Matthew Morrow\hfill {\tt morrow@math.uni-bonn.de}\\
Mathematisches Institut\hfill \url{http://www.math.uni-bonn.de/people/morrow/}\\\
Universit\"at Bonn\\
Endenicher Allee 60\\
53115 Bonn, Germany
\end{document}